\documentclass[reqno]{amsart}

\usepackage{amssymb,color}
\usepackage{amsmath}
\usepackage{amsfonts}
\usepackage{fouridx}
\usepackage[normalem]{ulem}
\usepackage{soul}
\usepackage[pdftex]{hyperref}
\usepackage{relsize}
\usepackage[utf8]{inputenc}
\usepackage{mathrsfs}
\usepackage[english]{babel}
\usepackage[totoc]{idxlayout}
\usepackage{fancyhdr}
\usepackage{paralist}
\usepackage{xcolor}
\usepackage[mathscr]{euscript}
\usepackage{mathtools}

\usepackage{todonotes}

\allowdisplaybreaks[3]
\newcommand\delc[1]{}

\newcommand\comcd[1]{}

\newcommand\del[1]{}
\newcommand\deln[1]{}
\newcommand\delr[1]{}

\newcommand\comad[1]{}

\newcommand\Greendel[1]{}
\newcommand\old[1]{}

\numberwithin{equation}{section}

\def\old#1{}

\def\text#1{{\rm #1}}
\def\newold#1{}

\theoremstyle{plain}

\numberwithin{equation}{section}

\begin{document}

\title{Existence of Martingale Solutions to Stochastic Constrained Heat Equation}

\author{Javed  Hussain}
\address{Department of Mathematics\\
Sukkur IBA University\\
Sindh Pakistan}
\email{javed.brohi@iba-suk.edu.pk}

\author{Abdul Fatah}
\address{Department of Mathematics\\
Sukkur IBA University\\
Sindh Pakistan} 
\email{afatah.msmaths21@iba-suk.edu.pk}

\author{Saeed Ahmed}
\address{Department of Mathematics\\
Sukkur IBA University\\
Sindh Pakistan}
\email{saeed.msmaths21@iba-suk.edu.pk}

\begin{abstract}
This article extends the work on stochastic constrained heat equation in \cite{brzezniak2020global}. We will show the existence of Martingale solutions to the stochastic-constrained heat equations. The proof is based on compactness, tightness of measure, quadratic variations, and Martingale representation theorem. 
\end{abstract}
\date{\today}
\maketitle



\newtheorem{theorem}{Theorem}[section]
\newtheorem{lemma}[theorem]{Lemma}
\newtheorem{proposition}[theorem]{Proposition}
\newtheorem{corollary}[theorem]{Corollary}
\newtheorem{question}[theorem]{Question}

\theoremstyle{definition}
\newtheorem{definition}[theorem]{Definition}
\newtheorem{algorithm}[theorem]{Algorithm}
\newtheorem{conclusion}[theorem]{Conclusion}
\newtheorem{problem}[theorem]{Problem}

\theoremstyle{remark}
\newtheorem{remark}[theorem]{Remark}
\numberwithin{equation}{section}

\section{Introduction}
\label{Introducing problem}
We are interested in martingale solutions to the following problem.

\begin{align}\label{SCHE}
    du&=\left[\Delta u +F(u)\right]dt+\sum_{j=1}^{N} \Lambda_{j}(u)\circ dW_{j} \\
    u(0)&=u_{0}
\end{align}

where $F(u)=||u||^{2}u+|u|_{L^{2n}}^{2n}-u^{2n-1}$ , $\Lambda_{j}(u)=\pi_j(u)=\varphi_{j}-\left  \langle \varphi_{j},u\right \rangle u$ and 
$\sum_{j=1}^{N} \Lambda_{j}(u)\circ dW_{j}$ is a Stratonovich type random forcing factor with $W_{j}$ being  independent $\mathbb{R}$-valued standard Brownian motions. In the above equation, noise in Stratonovich form is  gradient type  and is tangent to \\ $\mathcal{M}=\left\{u\in \mathcal{L}^2: |u|_{L^2}^2=1 \right\}$.\\
Interestingly, \eqref{SCHE} can also be written in It$\hat{o}$ form as below 
\begin{align} \label{main prob}
        du&=\left[\Delta u +F(u)+\frac{1}{2}\sum_{J=1}^{N} k_{j}(u) \right]dt+\sum_{j=1}^{N} \Lambda_{j}(u)dW_{j}, \\
    u(0)&=u_{0}.
\end{align}
with $ k_j(u)=d_{u}\Lambda_{j}(\Lambda_{j}(u))= -\left \langle \varphi_{j}, \Lambda_{j} (u) \right \rangle u-\left \langle \varphi_{j},u \right \rangle \Lambda_{j}(u) $.

\section{Important Function Spaces}
We will be working on the problem in the following function space settings. 

\begin{align}
    &\mathcal{H}=\mathcal{L}^{2}, \nonumber \\
    &\mathcal{V}=\mathcal{H}^1_{0}, \nonumber \\
&\mathcal{E}=D(\Delta)=D(A)=\mathcal{H}^1_{0}\cap \mathcal{H}^2. \nonumber
\end{align}
Indeed,
\begin{align}
    \mathcal{E} \hookrightarrow \mathcal{V} \hookrightarrow \mathcal{H}. \label{triplet}
\end{align}
\begin{itemize}
    \item Norm on $\mathcal{H}$ is represented by  $|\cdot|_{\mathcal{H}}\; \text{ or }\; |\cdot|$.
    \item Norm on $\mathcal{V}$ is represented by  $||\cdot||\text{ or }||\cdot||_{\mathcal{V}}$.
    \item Norm on $\mathcal{E}$ is represented by $|\cdot|_{\mathcal{E}} $.
\end{itemize}
 Hilbert manifold unit sphere in the space $\mathcal{H}$ can be described by 
 \begin{align*}
     \mathcal{M}:=\left\{u \in \mathcal{H}: |u|_{\mathcal{H}}^2=1 \right\}
 \end{align*}
 Additionally, we will need the following spaces:\\
 \begin{itemize}
     \item  $\mathcal{C}([0,T],\mathcal{H})$ with the norm $  \sup_{0\leq s \leq T} |u(t')|_{\mathcal{H}}, \text{ for } u \in \mathcal{C}([0,T],\mathcal{H})$.
     \item $L^2_{w}([0,T], D(A))$ denotes the space $\mathcal{L}^2([0,T],D(A))$ endowed\\ with the weak topology. 
     \item  $\mathcal{L}^2([0,T],\mathcal{V})$ with the norm $  ||u||_{ \mathcal{L}^2([0,T],\mathcal{V}) }= \left( \int_{0}^{T} ||u(t')||_{\mathcal{V}}^2 dt' \right)$.
     \item $\mathcal{C}([0,T],\mathcal{V}_{w})$ contains $\mathcal{V}_{w}$-valued continuous functions, in weak sense, on $[0,T]$. 
 \end{itemize}
 Throughout the discussion, we will denote our Laplace Operator by the usual notation $\Delta$ or $A$, i.e. $A(u)=\Delta(u)$.

\section{Important Definitions, lemmas and Theorems}
\begin{definition}
We say that there exists a martingale solution of \eqref{main prob} (SCHE) if there exist
\begin{enumerate}
    \item a stochastic basis
    $(\hat{\Omega},\hat{\mathcal{F}},\hat{F},\hat{P} )$
    \item an   $R^m$  valued  $\hat{F}$  Wiener process $\hat{W}$
    \item  and a  $\hat{F}--$ progressively measurable process $u : [0,T]\times\hat{\Omega}\rightarrow D(A)$ with $\hat{P}$ a.e.  paths $u(\cdot,w) \in \mathcal{C}([0,T],\mathcal{V}_{w})\cap \mathcal{L}^2([0,T],D(A))$
     such that for all  $t\in[0,T]$ and $v \in \mathcal{V}$  we have  
\end{enumerate}
\begin{align}
    \langle u(t), v \rangle- \langle u(0), v \rangle &=\int{ \langle \Delta u(s)+F(u(s)+\frac{1}{2}\sum_{j=1}^{N}\kappa _{j}(u(s)}, v \rangle ds+\nonumber \\ &\sum_{j=1}^{N}\int \langle B_{j}(u),v \rangle dW_{j}, 
\end{align} holds $P$ - a.s.  \cite{dhariwal2017study}
\end{definition}

\begin{definition}
    Suppose $v \in C ([0,T_1],\mathcal{S} )$ then we define \textbf{modulus of continuity of v} as 
    \begin{align}
        m(v,\epsilon)=\sup_{s_1,s_2 \in [0,T_1],\; |s_1-s_2|\leq \epsilon} \mathcal{\rho}(v(s_1),v(s_2)) \nonumber
    \end{align}
\end{definition}

\begin{theorem}\cite{dhariwal2017study}
    The sequence $Y_k$of random variables with values in $\mathcal{S}$satisfies \textbf{[T]} if and only if 
    \begin{align}
        \forall \epsilon_1 > 0,\epsilon_2 \; >0 \exists \delta > 0: \;\; \sup_{k \in \mathbb{N}}\mathbb{P}\left\{m(Y_k,\delta)>\epsilon_2\right\}\leq \epsilon_1  \nonumber
    \end{align}
\end{theorem}

\begin{lemma}
    Let $Y_k$ satisfies \textbf{[T]}. Assume $\mathbb{P}_k$ be the distribution/law of $Y_k$ on $C ([0,T],\mathcal{S})$, $k \in \mathbb{N}$. Then 
    \begin{align}
        \forall \epsilon' >0 , \exists A_{\epsilon'} \subset \mathcal{C} ([0,T],\mathcal{S}) : \; \; \sup_{k \in \mathbb{N}} \mathbb{P}_k(A_{\epsilon'}) + \epsilon' \geq 1 \text{ and } \lim_{\delta \rightarrow 0} \sup_{v \in A_{\epsilon}}(m(v,\delta))=0. \nonumber
    \end{align}
\end{lemma}

\begin{definition}(\textbf{Aldous Condition}) \cite{dhariwal2017study}
     The sequence $(Y_k)_{k \in \mathbb{N}}$satisfies Aldous condition \textbf{[A]} if and only if $\forall \epsilon_1>0, \epsilon_{2}>0 ,\exists \delta' > 0$ in such a way that for every sequence $s_k$of stopping times with $s_{k} \leq T $ we have :
    \begin{align}
        \sup_{k \in \mathbb{N}} \sup_{0\leq a \leq \delta  } \mathbb{P}(\mathcal{\rho}(Y_{k}(s_k+a),Y_{k}(s_k))\geq \epsilon_2)\leq \epsilon_{1} \nonumber
    \end{align}
\end{definition}

\begin{proposition}
    \textbf{[T]} and \textbf{[A]} are equivalent.
\end{proposition}

\begin{lemma}
    The space consisting of the continuous and square-integrable processes, which are martingales and have values in $\mathcal{H}$, is represented by $\mathbb{M}_{T}^{2}$. Moreover, it can be shown that $\mathbb{M}_{T}^{2}$ is a Banach space w.r.t following norm, 
    \begin{align*}
        \left \vert \right \vert g \left \vert \right \vert _{\mathbb{M}_{T}^{2}}=\sqrt{ \mathbb{E} \left(\sup_{ s\in [0,T]} ||g(s)||^{2}_{\mathcal{H}}\right)}.
    \end{align*}
\end{lemma}

\subsection{Martingale Representation Theorem}
\begin{theorem}\label{Martingale Repr THeorem}
   Consider Hilbert space $\mathcal{H}$ and $\mathcal{M} \in \mathcal{M}^{2}_{T}(\mathcal{H})$ with 
   \begin{align}
       \left\langle \mathcal{M} \right\rangle_{s}=\int_{0}^{s} \left(f(u)\mathbb{Q}^{\frac{1}{2}}\right)\cdot \left(f(u)\mathbb{Q}^{\frac{1}{2}}\right)^{*} du \nonumber
   \end{align}
   where $f(s)$ is a predictable process. \\ $\mathcal{H}_{0}=\mathbb{Q}^{\frac{1}{2}}\mathcal{H}$ is Hilbert space with following inner product
   \begin{align}
       \left\langle h, k \right\rangle_{\mathcal{H}_{0}}:= \left\langle \mathbb{Q}^{-\frac{1}{2}}h , \mathbb{Q}^{-\frac{1}{2}} k \right\rangle_{\mathcal{H}}, \;\; h,k \in \mathcal{H}_{0} \nonumber.
   \end{align}
And $\mathbb{Q}$ is a symmetric, bounded and non-negative operator in $\mathcal{H}$.
Then one can find the probability space $(\hat{\Omega},\hat{\mathcal{F}},\hat{\mathbb{P}})$, a filtration $\mathcal{F}_{s}$ and a $\mathcal{H}-$ valued $\mathbb{Q}-$ Wiener process $W$ defined on $\left(\Omega \times \hat{\Omega}, \mathcal{F} \times \hat{\mathcal{F}}, \mathbb{P} \times \hat{\mathbb{P}}\right)$ adapted to $\mathcal{F} \times \hat{\mathcal{F}}$ such that for all $(\omega,\hat{\omega}) \in \Omega \times \hat{\Omega}$,
\begin{align}
    \mathcal{M}(s,\omega,\hat{\omega})=\int_{0}^{s} f(u,\omega,\hat{\omega})d\omega(u,\omega,\hat{\omega}), \; s\in [0,T] ,
\end{align}
with
\begin{align}
    \mathcal{M}(s,\omega,\hat{\omega})=\mathcal{M}(s,\omega), \text{ and }\; f(s,\omega,\hat{\omega})=f(s,\omega).
\end{align}
\end{theorem}

\section{Faedo Galerkin Approximation} 
Let $ \left\{e_{k}\right\}_{k=1}^{\infty} $ be eigenvectors of 
$\Delta $ and also form orthonormal basis for the space $ \mathcal{H} $, then the finite dimensional subspace of $ \mathcal{H} $ can be defined as   
$ F_n:= \mathrm{span}\left\{e_1,..., e_n\right\} $ and a map $ Q_{n} : \mathcal{H} \rightarrow F_{n} $ by 
\begin{align*}
Q_{n}(u)=\sum_{k=1}^{n} \left \langle u , e_{k} \right \rangle_{\mathcal{H}} e_{k} 
\end{align*}.\\
In this way, the Faedo Galerkin Approximation of \eqref{SCHE} is given by 

\begin{align}\label{FAedo Galerkin befor}
    \begin{cases}
    du_{n}&=\left[Q_{n}\Delta(u_{n})+Q_{n}(F(u_{n}))\right]dt+\sum_{j=1}^{N} Q_{n}(\Lambda_{j}(u_{n}))\circ dW_{j}. \\ 
    u_{n}(0)&=\frac{Q_{n}(u_{0})}{|Q_{n}(u_{0})|}   
    \end{cases}
\end{align}
But from \cite{brzezniak2020global} we have $Q_{n}(\Delta(u_{n}))=\Delta(u_{n})$,
 $Q_{n}(F(u_{n}))=F(u_{n})$ and

 \begin{align*}
     Q_{n}(\Lambda_{j}(u_{n}))&=Q_{n}(\varphi_{j}-\left \langle \varphi_{j},u_{n} \right \rangle u_{n})\\
     &=Q_{n}(\varphi_{j})-\left \langle \varphi_{j},u_{n}\right \rangle Q_{n}(u_{n}) \\
     &=\varphi_{j}-\left \langle \varphi_{j}, u_{n} \right \rangle u_{n}\\
     &=\Lambda_{j}(u_{n}).
 \end{align*}
Hence, \eqref{FAedo Galerkin befor} takes the following form
\begin{align}\label{faedo after Qn}
    \begin{cases}
    du_{n}&=\left[\Delta(u_{n})+F(u_{n})\right]dt+\sum_{j=1}^{N} \Lambda_{j}(u_{n})\circ dW_{j}. \\ 
    u_{n}(0)&=\frac{Q_{n}(u_{0})}{|Q_{n}(u_{0})|}.   
    \end{cases}
\end{align}
And in It$\hat{o}$ form it is given by;

\begin{align}\label{Faedo in Ito}
    \begin{cases}
    du_{n}&=\left[\Delta(u_{n})+F(u_{n})+\frac{1}{2}\sum_{j=1}^{N} k_{j}(u_{n}))\right]dt+\sum_{j=1}^{N} \Lambda_{j}(u_{n})dW_{j}. \\ 
    u_{n}(0)&=\frac{Q_{n}(u_{0})}{|Q_{n}(u_{0})|}   
    \end{cases}
\end{align}.

\begin{theorem}\label{Invarience in Finite dimension}
If  $u_{0} \in \mathcal{M}\cap \mathcal{V}$ then $u(t)$ satisfying \eqref{faedo after Qn} belongs to $\mathcal{M}$.
\end{theorem}
\begin{proof}
 By letting $ \Phi(u_{n})=\frac{1}{2} |u_{n}|_{\mathcal{H}}^{2} $ and applying It$\hat{o}$ lemma to $ \Phi(u_{n})$ we get  \cite{brzezniak2020global}
\begin{align}
    |u_{n}(t)|_{\mathcal{H}}^2-1&= \sum_{j=1}^{N} \int_{0}^{t} 2\left \langle u_{n}(t'), \varphi_{j} \right \rangle \left(|u_{n}(t')|_{\mathcal{H}}^2-1\right)dW_{j}(t') \nonumber \\&+\int_{0}^{t} 2\left[||u_{n}(t')||^{2}+|u_{n}(t')|_{L^{2n}}^{2n}-|\varphi_{j}|_{\mathcal{H}}^{2}+3\left \langle u_{n}(t'), \varphi_{j}\right \rangle^{2}\right]\left(|u_{n}(t')|_{\mathcal{H}}^2-1\right)dt', \; \mathbb{P}-a.s. \nonumber
\end{align}
Setting:
\begin{align*}
\gamma(t)&:=|u_{n}(t)|_{\mathcal{H}}^2-1\\
a(t)&:=2\left \langle u_{n}(t'), \varphi_{j} \right \rangle\\
b(t)&:=2\left[|u_{n}(t')|_{L^{2n}}^{2n}+||u_{n}(t')||^{2}-|\varphi_{j}|_{\mathcal{H}}^{2}+3\left \langle u_{n}(t'), \varphi_{j}\right \rangle^{2}\right]\\
G_{1}(t,a(t))&:=a(t)\gamma(t)\\
G_{2}(t,a(t))&:=b(t)\gamma(t).
\end{align*}
We have 

\begin{align}\label{Invariance of manifold in finite dimenson}
\begin{cases}
     \gamma(t)=\int_{0}^{t} G_{1}(s,a(t')) dW(t')+\int_{0}^{t} G_{2}(s,b(t')) dt'.\\
     \gamma(0)=0
\end{cases}    
\end{align}
From \cite{brzezniak2020global}, a unique solution exists $w$ to \eqref{Invariance of manifold in finite dimenson}. Uniqueness gives 
\begin{align*}
    w(t)=\gamma(t)=0, \text{ i.e. }|u_{n}(t)|_{\mathcal{H}}^2=1.
\end{align*}
\end{proof}

\section{Some Important Estimates }
\begin{theorem}\label{Important Estimates}
Suppose $u_{n}(t)$ solves \eqref{faedo after Qn} then\\
\begin{align}
    \sup_{n\geq 1} \mathbb{E} \left(\sup_{s \in [0,T]}||u_{n}(t')||^{2}\right) \leq C_{1},\\
    \sup_{n\geq 1} \mathbb{E} \left(\sup_{s \in [0,T]}|u_{n}(t')|^{2n}_{{L^{2n}}}\right) \leq C_{2},\\
    \sup_{n\geq 1}\mathbb{E}\left(\int_{0}^{T} |u_{n}(t')|^{2}_{\mathcal{E}}dt'\right) \leq C_{3},
\end{align}
for some constants $C_{1},C_{2},C_{3}>0$.
\end{theorem}
\begin{proof}
 Let 
\begin{align}
    \phi(u)=\frac{1}{2}||u_{n}(t')||^2+\frac{1}{2}|u_{n}(t')|^{2n}_{{L^{2n}}} \nonumber
\end{align}
then applying It$\hat{o}$ lemma to $\phi(u)$ it has been proved in \cite{brzezniak2020global} that expectation $\mathbb{E}(\phi(u_{n}(t)))$ is uniformly bounded.\\ Hence, we have $C_{1}$ and $C_{2}$ such that
\begin{align}
    \sup_{n\geq 1} \mathbb{E} \left(\sup_{s \in [0,T]}||u_{n}(t')||^{2}\right) \leq C_{1} \label{un in V} \nonumber\\
    \sup_{n\geq 1} \mathbb{E} \left(\sup_{s \in [0,T]}|u_{n}(t')|^{2n}_{{L^{2n}}}\right) \leq C_{2} \nonumber
\end{align}
In order to show \\
\begin{align}
    \sup_{n\geq 1}\mathbb{E}\left(\int_{0}^{T} |u_{n}(t')|^{2}_{\mathcal{E}}dt'\right) \leq C_{3}.
\end{align}
We will apply It$\hat{o}$ lemma to $\Gamma(u)$ where
\begin{align}
    \Gamma(u)=\frac{1}{2}||u||^2
\end{align}
In order to apply It$\hat{o}$ lemma we will make use of following terms computed in \cite{brzezniak2020global}.
\begin{align}
D_{u}(\Gamma(u))=\left \langle \Gamma'(u), h \right \rangle=\left \langle u, h \right \rangle_{\mathcal{V}} \nonumber \\
D^2_{u}(\Gamma(u))=\left \langle \Gamma''(u)a_{1}, a_{2} \right \rangle=\left \langle a_{1}, a_{2} \right \rangle_{\mathcal{V}}\nonumber.
\end{align}
In this perspective, 
after applying It$\hat{o}$ lemma to $\Gamma(u)$ yieldt' 
\begin{align}\label{Ito Lemma }
    \Gamma(u(t))-\Gamma(u(0))&=\sum_{j=1}^{N} \int_{0}^{t} \left \langle \Gamma'(u\left(t'\right)), \Lambda_{j}(u\left(t'\right)) \right \rangle dW_{j}\left(t'\right) +\int_{0}^{t} \left \langle \Gamma'(u\left(t'\right)), \Delta(u\left(t'\right))+F(u\left(t'\right)) \right \rangle dt'\nonumber\\&+\frac{1}{2} \sum_{j=1}^{N} \int_{0}^{t} \left \langle \Gamma'(u\left(t'\right)), k_{j}(u\left(t'\right)) \right \rangle dt'+\frac{1}{2} \sum_{j=1}^{N} \int_{0}^{t}\left \langle \Gamma''(u)\Lambda_{j}(u\left(t'\right)),\Lambda_{j}(u\left(t'\right)) \right \rangle dt'.
\end{align}
But
\begin{align}\label{gamma with delta and G}
    \left \langle \Gamma'(u), \Delta(u)+F(u) \right \rangle &=\left \langle u, \Delta(u)+F(u) \right \rangle_{\mathcal{V}}=\left \langle u, \Delta(u) \right \rangle_{\mathcal{V}}+\left \langle u, F(u) \right \rangle_{\mathcal{V}} \nonumber \\
    &=-\left \langle \nabla u, \nabla u \right \rangle_{\mathcal{V}}+\left \langle u, G(u) \right \rangle_{\mathcal{V}} =-||\nabla u||^{2}_{\mathcal{V}}+\left \langle u, F(u) \right \rangle_{\mathcal{V}} \nonumber \\
    &=-|| u||^{2}_{\mathcal{E}}+\left \langle u, F(u) \right \rangle_{\mathcal{V}}.
\end{align}
Equation \eqref{Ito Lemma } along with
equation \eqref{gamma with delta and G} produces 
\begin{align}\label{Ito Lemma again }
    \Gamma(u(t))-\Gamma(u(0))&=\sum_{j=1}^{N} \int_{0}^{t} \left \langle \Gamma'(u(t')), \Lambda_{j}(u(t')) \right \rangle dW_{j}(t') -\int_{0}^{t}|| u||^{2}_{\mathcal{E}} dt'+\int_{0}^{t} \left \langle u, F(u) \right \rangle_{\mathcal{V}} dt' \nonumber\\&+\frac{1}{2} \sum_{j=1}^{N} \int_{0}^{t} \left \langle \Gamma'(u(t')), k_{j}(u(t')) \right \rangle dt'+\frac{1}{2} \sum_{j=1}^{N} \int_{0}^{t}\left \langle \Gamma''(u)\Lambda_{j}(u(t')),\Lambda_{j}(u(t')) \right \rangle dt'\nonumber\\
    &=\sum_{0}^{N} K_{1,j}-\int_{0}^{t}|| u||^{2}_{\mathcal{E}} dt'+K_{2}+\frac{1}{2}\sum_{1}^{N}K_{3,j}+\frac{1}{2}\sum_{1}^{N}K_{4,j}\nonumber\\
    \Gamma(u(t))+\int_{0}^{t}|| u||^{2}_{\mathcal{E}} dt'&=\Gamma(u(0))+\sum_{0}^{N} K_{1,j}+K_{2}+\frac{1}{2}\sum_{1}^{N}K_{3,j}+\frac{1}{2}\sum_{1}^{N}K_{4,j}. \nonumber
\end{align}
Where
\begin{align*}
    K_{1,j}&=\int_{0}^{t} \left \langle \Gamma'(u(t')), \Lambda_{j}(u(t')) \right \rangle dW_{j}(t')\\
    K_{2}&=\int_{0}^{t} \left \langle u, F(u) \right \rangle_{\mathcal{V}} dt'\\
    K_{3,j}&=\int_{0}^{t} \left \langle \Gamma'(u(t')), k_{j}(u(t')) \right \rangle\\
    K_{4,j}&=\int_{0}^{t}\left \langle \Gamma''(u)\Lambda_{j}(u(t')),\Lambda_{j}(u(t')) \right \rangle dt'
\end{align*}
We will show uniformly boundedness of each of the above term one by one.

In \cite{brzezniak2020global}, it has been shown that
\begin{align}
    \mathbb{E}\left(\int_{0}^{t} \left \langle \Gamma'(u(t')), \Lambda_{j}(u(t'))^{2} \right \rangle dt'\right) < \infty   \nonumber
\end{align}
which makes $K_{1,j}$ martingale thus
\begin{align}
    \mathbb{E}(K_{1,j})=0.
\end{align}
Subsequently,
\begin{align}
    K_{2}&=\int_{0}^{t} \left \langle u(t') ,F(u(t')) \right \rangle_{\mathcal{V}}\nonumber\\
    &=\int_{0}^{t} \left \langle u(t') ,||u(t')||^{2}u(t')+|u(t')|_{L^{2n}}^{2n}u-u^{2n-1}(t') \right \rangle_{\mathcal{V}}dt'\nonumber\\
    &=\int_{0}^{t} ||u(t')||^{4}+|u(t')|_{L^{2n}}^{2n}||u(t')||^{2}-\left \langle u(t'),u^{2n-1}(t') \right \rangle_{\mathcal{V}}dt'\nonumber\\
    &=\int_{0}^{t} ||u(t')||^{4}dt'+\int_{0}^{t}|u(t')|_{L^{2n}}^{2n}||u(t')||^{2}dt'-\int_{0}^{t}\left \langle u(t'),u^{2n-1}(t') \right \rangle_{\mathcal{V}}dt'\nonumber\\
    &=K_{21}+K_{22}-K_{23}. \label{K2 bound}
\end{align}

With the help of theorem \eqref{Important Estimates} $K_{21}$ satisfies following estimate
\begin{align}\label{K21 estimate}
    K_{21}<C^{4}_{1}T.
\end{align}
Taking $K_{22}$,
\begin{align}
K_{22}=\int_{0}^{t}|u(t')|_{L^{2n}}^{2n}||u(t')||^{2}dt',\nonumber
\end{align}
in the light of Gagliardo Nirenberg inequality we have $ \mathcal{V} \hookrightarrow L^{2n} $ 
\begin{align}\label{K22 estimate}
    K_{22}&\leq C^{2n} \int_{0}^{t} ||u(t')||^{2n+2}_{\mathcal{V}} dt' 
    \leq C^{2n}C_{1}^{2n+2}t 
    \leq C^{2n}C_{1}^{2n+2}T.
\end{align}
Now we work with $K_{23}$
\begin{align}
K_{23}&=\int_{0}^{t} \left \langle u(t'),u^{2n-1}(t')
\right \rangle_{\mathcal{V}} dt'\nonumber\\
K_{23}&=\int_{0}^{t} \left(\int_{\Omega} \nabla u(t')\cdot \nabla u^{2n-1}(t')dx\right) dt'\nonumber\\
K_{23}&=\int_{0}^{t} \left((2n-1) \int_{\Omega} |\nabla u(t')|^{2}\cdot u^{2n-2}(t')dx\right) dt'\nonumber\\
K_{23}&<(2n-1)\int_{0}^{t} \left( \int_{\Omega} |\nabla u(t')|^{4}dx\right)^{\frac{1}{2}}\left(\int_{\Omega} u^{4n-4}(t')dx\right)^{\frac{1}{2}} dt'\nonumber\\
K_{23}&<(2n-1)\int_{0}^{t} \left(|\nabla u(t')|^{2}_{\mathcal{L}^4}\right)\left(|u|^{2n-2}_{L^{4n-4}}\right) dt'\nonumber
\end{align}
Using $L^p$ embeddings we achieve,
\begin{align}
    K_{23}&\leq(2n-1)\int_{0}^{t} C^{2}_{4} |u(t')|_{\mathcal{H}}^{2}C_{5}^{2n-2}|u(t')|^{2n-2}_{\mathcal{H}} dt'\nonumber\\
    &\leq(2n-1)C^{2}_{4}C_{5}^{2n-2}\int_{0}^{t}  |u(t')|_{\mathcal{H}}^{2n}dt'\nonumber\\
    &\leq(2n-1)C^{2}_{4}C_{5}C_{6}^{2n-2}\int_{0}^{t}  ||u(t')||_{\mathcal{V}}^{2n}dt'\nonumber\\
    &\leq(2n-1)C^{2}_{4}C_{5}C_{6}^{2n-2}C_{1}^{2n}T \label{K33 estimate}.
\end{align}

Equation \eqref{K2 bound} in the light of \eqref{K21 estimate}, \eqref{K22 estimate} and \eqref{K33 estimate} implies,
\begin{align}\label{k2 estimate}
    K_{2}< C_{7}.
\end{align}
For $C_{7}:=C^{4}_{1}T+C^{2n}C_{1}^{2n+2}T+(2n-1)C^{2}_{4}C_{5}C_{6}^{2n-2}C_{1}^{2n}T$.\\
Considering $K_{3,j}$,
\begin{align}
    K_{3,j}&=\int_{0}^{t} \left \langle u\left(t'\right), k(u\left(t'\right))\right \rangle_{\mathcal{V}}dt'\nonumber\\
    K_{3,j}&=\int_{0}^{t} \left \langle u\left(t'\right), -\left \langle \varphi_{j}, \Lambda_{j}(u)\right \rangle_{\mathcal{H}} u\left(t'\right)-\left \langle \varphi_{j} , u\left(t'\right) \right \rangle_{\mathcal{H}} \Lambda_{j}(u\left(t'\right)) \right \rangle_{\mathcal{V}}dt'\nonumber\\
    K_{3,j}&=\int_{0}^{t} -(\left \langle u\left(t'\right), u\left(t'\right)\right \rangle_{\mathcal{V}})\left \langle \varphi_{j}, \Lambda_{j}(u)\right \rangle_{\mathcal{H}} -\left \langle \varphi_{j} , u\left(t'\right) \right \rangle_{\mathcal{H}} \left \langle u\left(t'\right), \Lambda_{j}(u\left(t'\right)) \right \rangle_{\mathcal{V}}dt'\nonumber\\
    K_{3,j}&=\int_{0}^{t} -\vert\vert u\left(t'\right)\vert\vert^{2}_{\mathcal{V}}\left \langle \varphi_{j}, \Lambda_{j}(u)\right \rangle_{\mathcal{H}} -\left \langle \varphi_{j} , u\left(t'\right) \right \rangle_{\mathcal{H}} \left \langle u\left(t'\right), \Lambda_{j}(u\left(t'\right))  \right \rangle_{\mathcal{V}} dt'\nonumber\\
    &\leq \left \vert \int_{0}^{t} -\vert\vert u\left(t'\right)\vert\vert^{2}_{\mathcal{V}}\left \langle \varphi_{j}, \Lambda_{j}(u)\right \rangle_{\mathcal{H}} -\left \langle \varphi_{j} , u\left(t'\right) \right \rangle_{\mathcal{H}} \left \langle u\left(t'\right), \Lambda_{j}(u\left(t'\right))  \right \rangle_{\mathcal{V}} dt' \right\vert \nonumber\\
    &\leq \int_{0}^{t} \left \vert  -\vert\vert u\left(t'\right)\vert\vert^{2}_{\mathcal{V}}\left \langle \varphi_{j}, \Lambda_{j}(u)\right \rangle_{\mathcal{H}} -\left \langle \varphi_{j} , u\left(t'\right) \right \rangle_{\mathcal{H}} \left \langle u\left(t'\right), \Lambda_{j}(u\left(t'\right))  \right \rangle_{\mathcal{V}} \right\vert dt'  \nonumber\\
    &\leq \int_{0}^{t} \left \vert \vert\vert u\left(t'\right)\vert\vert^{2}_{\mathcal{V}}\left \langle \varphi_{j}, \Lambda_{j}(u)\right \rangle_{\mathcal{H}}\right \vert + \left \vert \left \langle \varphi_{j} , u\left(t'\right) \right \rangle_{\mathcal{H}} \left \langle u\left(t'\right), \Lambda_{j}(u\left(t'\right))  \right \rangle_{\mathcal{V}} \right\vert dt'  \nonumber\\
    &\leq \int_{0}^{t}  \vert\vert u\left(t'\right)\vert\vert^{2}_{\mathcal{V}} \left \vert \left \langle \varphi_{j}, \Lambda_{j}(u)\right \rangle_{\mathcal{H}}\right \vert + \left \vert \left \langle \varphi_{j} , u\left(t'\right) \right \rangle_{\mathcal{H}} \right \vert  \left \vert \left \langle u\left(t'\right), \Lambda_{j}(u\left(t'\right))  \right \rangle_{\mathcal{V}} \right\vert dt'  \nonumber\\
    &\leq \int_{0}^{t}  \left\vert\right\vert u\left(t'\right) \left\vert\right\vert^{2}_{\mathcal{V}} \left\vert\right\vert \varphi_{j} \left\vert\right\vert^{2}_{\mathcal{H}}  \left\vert\right\vert \Lambda_{j}(u)\left\vert\right\vert^{2}_{\mathcal{H}} + \left\vert\right\vert  \varphi_{j} \left\vert\right\vert^{2}_{\mathcal{H}} \left \vert \vert  u\left(t'\right) \right \vert  \vert^{2}_{\mathcal{H}}\left \vert \left \langle u\left(t'\right), \Lambda_{j}(u\left(t'\right))  \right \rangle_{\mathcal{V}} \right\vert dt'  \nonumber\\
    &\leq \int_{0}^{t}  \left\vert\right\vert u\left(t'\right) \left\vert\right\vert^{2}_{\mathcal{V}} \left\vert\right\vert \varphi_{j}\left\vert\right\vert^{2}_{\mathcal{H}}  \left\vert\right\vert \Lambda_{j}(u)\left\vert\right\vert^{2}_{\mathcal{H}} + \left\vert\right\vert  \varphi_{j} \left\vert\right\vert^{2}_{\mathcal{H}} \left \vert \vert  u\left(t'\right) \right \vert  \vert^{2}_{\mathcal{H}} \left \vert \vert  u\left(t'\right) \right \vert  \vert^{2}_{\mathcal{V}} \left \vert \right \vert  \Lambda_{j}(u\left(t'\right)) \left \vert\right \vert  \vert^{2}_{\mathcal{V}} dt'. \label{k3j estimate befor Bj , u and fj}
\end{align}
But
\begin{align}
\left\vert \right \vert \Lambda_{j}(u) \left\vert \right \vert^{2}&=\left\vert \right \vert f-\left \langle \varphi_{j} u\right \rangle_{\mathcal{H}} u \left\vert \right \vert^{2}\nonumber\\
\left\vert \right \vert \Lambda_{j}(u) \left\vert \right \vert^{2}& \leq \left\vert \right \vert \varphi_{j} \left\vert \right \vert^{2} + \left\vert  \varphi_{j} \right \vert^{2}_{\mathcal{H}} \left\vert  u \right \vert^{2}_{\mathcal{H}}  \left\vert \right \vert u \left\vert \right \vert^{2} \label{ Bj estimate}
\end{align}
As
\begin{align} \label{fj is in V and H}
    \varphi_{j} \in \mathcal{V} \hookrightarrow \mathcal{H},
\end{align}
so it follows that,
\begin{align}\label{u bounded in V}
    \left \vert \right \vert u \left \vert \right \vert^{2} < C_{1}.
\end{align}

Combining \eqref{k3j estimate befor Bj , u and fj}, \eqref{ Bj estimate} and \eqref{fj is in V and H}  together with  \eqref{u bounded in V}, we can find a constant $C_{8}$ such that  
\begin{align} \label{k3j estimate}
    K_{3j} \leq C_{8}.
\end{align}
Finally, we will consider 

\begin{align}
    K_{4,j}&=\int_{0}^{t}\left \langle \Gamma''(u)\Lambda_{j}(u(t')),\Lambda_{j}(u(t')) \right \rangle dt' =\int_{0}^{t} \left\vert \right \vert B_{u(t')} \left \vert \right \vert^{2}  dt'.  \label{k4j before estimate}
\end{align}
\eqref{k4j before estimate} in the view of \eqref{ Bj estimate}, \eqref{fj is in V and H} and \eqref{u bounded in V}  assures an existence of a constant say $C_{9}$  such that 
\begin{align}\label{k4j estimate}
    K_{4j} \leq C_{9}.
\end{align}

Combining \eqref{k4j estimate}, \eqref{k3j estimate} and \eqref{k2 estimate} with \eqref{Ito Lemma again } we get 
\begin{align}
    \frac{1}{2}\left\vert \right\vert u \left\vert \right\vert^{2}+\int_{0}^{t}  \left\vert \right\vert u \left\vert \right\vert^{2}_{\mathcal{E}} \leq \frac{1}{2} \left\vert \right\vert u_{0} \left\vert \right\vert^{2} +\sum_{j=0}^{N} K_{1j} +C_{7}+\frac{1}{2} \sum_{j=1}^{N} C_{8}+\frac{1}{2} \sum_{j=1}^{N} C_{9}. \nonumber   
\end{align}
In particular we have,

\begin{align*}
    \int_{0}^{T}  \left\vert \right\vert u_{n} \left\vert \right\vert^{2}_{\mathcal{E}} \leq \frac{1}{2} \left\vert \right\vert u_{0} \left\vert \right\vert^{2} +\sum_{j=0}^{N} K_{1j} +C_{7}+\frac{1}{2} \sum_{j=1}^{N} C_{8}+\frac{1}{2} \sum_{j=1}^{N} C_{9}   
\end{align*}
Applying expectation, both sides 

\begin{align*}
    \mathbb{E} \left( \int_{0}^{T}  \left\vert \right\vert u_{n} \left\vert \right\vert^{2}_{\mathcal{E}} \right)\leq \frac{1}{2} \left\vert \right\vert u_{0} \left\vert \right\vert^{2} +\sum_{j=0}^{N} \mathbb{E}(K_{1j}) +C_{7}+\frac{1}{2} \sum_{j=1}^{N} C_{8}+\frac{1}{2} \sum_{j=1}^{N} C_{9}   
\end{align*}

\begin{align*}
    \mathbb{E} \left( \int_{0}^{T}  \left\vert \right\vert u_{n} \left\vert \right\vert^{2}_{\mathcal{E}} \right)\leq \frac{1}{2} \left\vert \right\vert u_{0} \left\vert \right\vert^{2} +C_{7}+\frac{1}{2} \sum_{j=1}^{N} C_{8}+\frac{1}{2} \sum_{j=1}^{N} C_{9}   
\end{align*}
Set $C_{3}:= \frac{1}{2} \left\vert \right\vert u_{0} \left\vert \right\vert^{2} +C_{7}+\frac{1}{2} \sum_{j=1}^{N} C_{8}+\frac{1}{2} \sum_{j=1}^{N} C_{9}< \infty$. We have

\begin{align*}
    \mathbb{E} \left( \int_{0}^{T}  \left\vert \right\vert u_{n} \left\vert \right\vert^{2}_{\mathcal{E}} \right) < C_{3}
\end{align*}
Taking sup over $n$,
\begin{align}\label{u bounded in E}
    \sup_{n\geq 1}\mathbb{E} \left( \int_{0}^{T}  \left\vert \right\vert u_{n}(t') \left\vert \right\vert^{2}_{\mathcal{E}} dt' \right) < C_{3}
\end{align}
Thus proof gets concluded.
\end{proof}

\section{Tightness of Measures}\label{tightness of measure}
\begin{theorem}\label{set of measures tight thm}
    Let $\left\{\mathcal{L}(u_{n}): n\in \mathbf{N}\right\}$ be a set of measures, then this set is tight on $(\mathcal{Y}_{T}, \mathcal{F}))$. Here
\begin{align*}
\mathcal{Y}_{T}=C{([0,T],\mathcal{L}^{2});\mathcal{H})}\cup L_{w}^2([0,T];D(\Delta))\cup \mathcal{L}^{2}([0,T];\mathcal{V})\cup \mathcal{C}([0,T],\mathcal{V}_{w}).
\end{align*}
\end{theorem}
\begin{proof}
We guarantee that the sequence obeys Aldous conditions, which will suffice as proof. Consider the sequence $(\tau_n)_{n \in \mathbf{N}}$ of stopping times in such a way that $0\leq s_{n} \leq T $. Taking equation \eqref{Faedo in Ito}, and $0\leq t\leq T$,
\begin{align*}
    u_{n}&=u_{n}(0) +\int_{0}^{t} \Delta u_{n}(t') dt'+ \int_{0}^{t} F(u_{n}(t')) dt' +\frac{1}{2}\sum_{j=1}^{N} \int_{0}^{t} k_{j}(u_{n}(t')))dt'+\sum_{j=1}^{N} \int_{0}^{t} \Lambda_{j}(u_{n}(t'))dW_{j} \\ &:=L_{1}^{n}+L_{2}^{n}+L_{3}^{n}+.5\sum_{j=1}^{N} L_{4}^{n}+\sum_{j=1}^{N} L_{5}^{n}
\end{align*}
consider  $a > 0$. We will estimate $L_{i}^{n}$.
Since $\Delta: \mathcal{E} \rightarrow \mathcal{H}$ is linear and continuous, so by applying Holder's inequality and \eqref{u bounded in E}
\begin{align}
    \mathbb{E} \left(\left|L^{n}_{2}(s_{n}+a)-L^{n}_{2}(s_{n})\right|_{\mathcal{H}}\right)&=\mathbb{E} \left(\left|\int_{s_{n}}^{s_{n}+a} \Delta u_{n}(t') dt'\right|_{\mathcal{H}}\right)\nonumber \\ 
    &\leq \mathbb{E} \left(\int_{s_{n}}^{s_{n}+a} \left|\Delta u_{n}(t') \right|_{\mathcal{H}}dt'\right)\nonumber\\ 
    &\leq \mathbb{E} \left(\int_{s_{n}}^{s_{n}+a} ||u_{n}(t') ||_{\mathcal{E}}dt'\right)\nonumber\\
    &\leq ka^{\frac{1}{2}} \left ( \mathbb{E} \left(\int_{0}^{T} ||u_{n}(t') ||^{2}_{\mathcal{E}}dt'\right) \right)^{\frac{1}{2}}\nonumber\\
    &\leq ka^{\frac{1}{2}}C^{\frac{1}{2}}  \label{L2 estimate}
\end{align}

Taking $L_{3}^n$
\begin{align}
    \mathbb{E} (\left|L^{n}_{3}\left(s_{n}+a)-L^{n}_{3}(s_{n})\right|_{\mathcal{H}}\right)&=\mathbb{E} \left(\left|\int_{s_{n}}^{s_{n}+a} F( u_{n}(t')) dt'\right|_{\mathcal{H}}\right)\nonumber\\
    &\leq \mathbb{E} \left(\int_{s_{n}}^{s_{n}+a} |F( u_{n}(t')) |_{\mathcal{H}}dt'\right) \label{Estimating L3 before G(u) estimate}
\end{align}
From \cite{brzezniak2020global} $F(u_n)$ satisfies, 
\begin{align}
    |F(u_{n})-F(v_{n})|_{\mathcal{H}} \leq H_{1}(||u_{n}||,||v_{n}||) ||u_{n}-v_{n}|| \nonumber
\end{align}
so with $v_n=0$
\begin{align}\label{estimate for G(u)}
    |F(u_{n})|_{\mathcal{H}} \leq H_{1}(||u_{n}||,0) ||u_{n}||.
\end{align}

\eqref{Estimating L3 before G(u) estimate} and \eqref{estimate for G(u)} together implies that

\begin{align}
    \mathbb{E} (|L^{n}_{3}\left(s_{n}+a)-L^{n}_{3}(s_{n})|_{\mathcal{H}}\right) &\leq \mathbb{E} \left(\int_{s_{n}}^{s_{n}+a} H_{1}(||u_{n}||,0) ||u_{n}|| dt'\right) \nonumber \\
    &\leq \mathbb{E} \left(\int_{0}^{T} H_{1}(||u_{n}(t')||,0) ||u_{n}(t')|| dt'\right) \nonumber\\
    &\leq \mathbb{E} \left(\int_{0}^{T} H_{1}\left(||u_{n}(t')||,0 \right
    ) ||u_{n}(t')|| dt'\right). \label{before using Hn}
\end{align}
One can find in \cite{brzezniak2020global} that $H_1$ is given by
\begin{align*} 
H_{1}(||u_{n}||,||v_{n}||)&=C\left[||u_{n}||^{2}+||v_{n}||^{2}+\left(||u_{n}||+||v_{n}||\right)^2\right]+\\ &C_{2n}\left[ \left(\frac{2n-1}{2}\right)(||u_{n}||^{2n-1}+||v_{n}||^{2n-1})(||u_{n}||+||v_{n}||)\right]\\&+C_{2n}\left[\left(||u_{n}||^{2n}+||v_{n}||^{2n}\right)+\left(||u_{n}||^{2n-2}+||v_{n}||^{2n-2}\right) \right].   
\end{align*}
Hence,
\begin{align*} H_{1}(||u_{n}||,0)&=C\left[2||u_{n}||^{2}\right]+C_{2n}\left[ \left(\frac{2n-1}{2}\right)\left(||u_{n}||^{2n-1}\right)(||u_{n}||)+\left(||u_{n}||^{2n}\right)+\left(||u_{n}||^{2n-2}\right) \right]\\
&=2C\left[||u_{n}||^{2}\right]+C_{2n}\left[ \left(\frac{2n-1}{2}\right)\left(||u_{n}||^{2n}\right)+\left(||u_{n}||^{2n}\right)+\left(||u_{n}||^{2n-2}\right) \right]\\
&=2C||u_{n}||^{2}+C_{2n}\left[ (\frac{2n+1}{2})||u_{n}||^{2n}+||u_{n}||^{2n-2} \right],
\end{align*}
with the help of Theorem \eqref{Important Estimates} we can find a constant $C_{10}$ such that 
\begin{align}
    H_{1}(||u||,0)\leq C_{10}. \label{bounded for H1 of u , 0}
\end{align}
This way, \eqref{before using Hn} along with \eqref{bounded for H1 of u , 0} implies,
\begin{align*}
    \mathbb{E}\left(\int_{0}^{T} H_{1}(||u_{n}(t')||,0)||u_{n}(t')|| dt' \right)\leq     &\mathbb{E} \left(\int_{0}^{T} C_{10}||u_{n}(t')||dt'\right)\\
    &\leq \mathbb{E} \left[ \left(\int_{0}^{T} C_{10}^{2} dt'\right)^{\frac{1}{2}}\cdot \left(\int_{0}^{T} ||u_{n}(t')||^{2} dt'\right)^{\frac{1}{2}}     \right]\\
    &=C_{10}T^{\frac{1}{2}}\mathbb{E} \left(\int_{0}^{T} ||u_{n}(t')||^{2} dt'\right)^{\frac{1}{2}}\\
    &\leq C_{10}T^{\frac{1}{2}} \left(\mathbb{E} \left[\int_{0}^{T} ||u_{n}(t')||^{2} dt'\right]\right)^{\frac{1}{2}}
    \\&\leq C_{10}T^{\frac{1}{2}} \left(\mathbb{E} \left[\int_{0}^{T} C_{1} dt'\right]\right)^{\frac{1}{2}}\\
    &=C_{1}^{\frac{1}{2}}C_{10}T.
\end{align*}
Therefore, it follows that,
\begin{align}
    \mathbb{E} (|L^{n}_{3}\left(s_{n}+a)-L^{n}_{3}(s_{n})|_{\mathcal{H}}\right) \leq C_{1}^{\frac{1}{2}}C_{10}^{2}T. \label{L3 estimate}
\end{align}
\\
Let's work out $\mathbb{E} (|L^{n}_{4}\left(s_{n}+a)-L^{n}_{4}(s_{n})|_{\mathcal{H}}\right)$

\begin{align}
    \mathbb{E} (|L^{n}_{4}\left(s_{n}+a)-L^{n}_{4}(s_{n})|_{\mathcal{H}}\right)&=\mathbb{E} \left\vert \int_{s_{n}}^{s_{n}+a} k_{j}(u_{n}(t')))dt' \right\vert_{\mathcal{H}}\nonumber\\
    &\leq \mathbb{E}  \int_{s_{n}}^{s_{n}+a} \left\vert k_{j}(u_{n}(t'))) \right\vert_{\mathcal{H}}dt' \label{kj bfore estimate}
\end{align}
One can see in \cite{brzezniak2020global} that,
\begin{align}
    \left\vert \left\vert k_{j}(u_{n})-k_{j}(v_{n})\right\vert\right\vert \leq C ||\varphi_{j}||^{2} \left[2+ ||u_{n}||^{2}+||v_{n}||^{2} +\left(||u_{n}||+||v_{n}||\right)^{2}\right]||u_{n}-v_{n}||. \label{kj from brz}
\end{align}
Inequality \eqref{kj from brz} with $v_{n}=0$ and Young's Inequality produces,
\begin{align}
\left\vert \left\vert k_{j}(u_{n})\right\vert\right\vert &\leq C ||\varphi_{j}||^{2} \left[2+ ||u_{n}||^{2}+||u_{n}||^{2}\right]||u_{n}|| \nonumber\\
&=C ||\varphi_{j}||^{2} \left[2||u_{n}||+ 2||u_{n}||^{3}\right] \nonumber\\
&=2C ||\varphi_{j}||^{2}||u_{n}||+ 2C ||\varphi_{j}||^{2}||u_{n}||^{3} \nonumber\\
&\leq \frac{4C^2||\varphi_{j}||^{4}}{2}+\frac{||u_{n}||^2}{2}+ \frac{4C^2 ||\varphi_{j}||^{4}}{2}+\frac{(||u_{n}||^2)^{3}}{2}\nonumber\\
& \leq 4C^2 ||\varphi_{j}||^{4}+\frac{||u_{n}||^2}{2}+\frac{(||u_{n}||^2)^{3}}{2}, \nonumber
\end{align}
again using Theorem \eqref{Important Estimates} , $\mathcal{V} \hookrightarrow \mathcal{H} $ 
 and $||\varphi_{j}||<\infty $ makes sure the existence of constant say $C_{11}$ such that

\begin{align}
\left\vert k_{j}(u)\right\vert_{\mathcal{H}} \leq C_{11}. \label{kj estimate} 
\end{align}
Inequality \eqref{kj bfore estimate} together with inequality \eqref{kj estimate} results in;
\begin{align}
    \mathbb{E} \left(\left|L^{n}_{4}\left(s_{n}+a\right)-L^{n}_{4}(s_{n})\right|_{\mathcal{H}}\right) \leq a C_{11}. \label{L4 estimate}
\end{align}
\\
Finally, we need to estimate
$\mathbb{E}(|L^{n}_{5}\left(s_{n}+a)-L^{n}_{5}(s_{n})|^2_{\mathcal{H}}\right)$.
With the application of It$\hat{o}$ isometry we have;
\begin{align}
\mathbb{E} (|L^{n}_{5}\left(s_{n}+a)-L^{n}_{5}(s_{n})|^2_{\mathcal{H}}\right)&= \mathbb{E} \left(\left\vert \int_{s_{n}}^{s_{n}+a} \Lambda_{j}(u_{n}(t'))dW_{j} \right\vert^2 \right)\nonumber\\
& \leq  \mathbb{E} \left( \int_{s_{n}}^{s_{n}+a} \left\vert \Lambda_{j}(u_{n}(t'))\right\vert^2 dt'  \right). \label{L5 before Bj of un estimate}
\end{align}
However,
\begin{align}
    \left \vert \left \vert \Lambda_{j}(u_{n})-B_j(v_{n})\right\vert\right\vert \leq ||\varphi_{j}||\left(||u_{n}||+||v_{n}||\right)||u_{n}-v_{n}||. \nonumber
\end{align}
Setting $v=0$ and making use of Theorem\eqref{Important Estimates} derives,
\begin{align}
    &\left \vert \left \vert \Lambda_{j}(u_{n})-B_j(0)\right\vert\right\vert \leq ||\varphi_{j}||||u_{n}||^{2}\nonumber\\
    &\left \vert \left \vert \Lambda_{j}(u_{n})-\varphi_{j}\right\vert\right\vert \leq ||\varphi_{j}||||u_{n}||^{2}\leq ||\varphi_{j}|| C_{1}\label{B of u nj estimate}.
\end{align}
Because $\mathcal{V} \hookrightarrow \mathcal{H}$, so 
\begin{align}
    \left\vert \Lambda_{j}(u_{n})\right\vert_{\mathcal{H}} &\leq  C\left\vert \left\vert \Lambda_{j}(u_{n})\right\vert\right\vert, \nonumber\\
    &=C\left\vert \left\vert \Lambda_{j}(u_{n})-\varphi_{j}+\varphi_{j}\right\vert\right\vert,\nonumber\\
    &\leq C\left\vert \left\vert \Lambda_{j}(u_{n})-\varphi_{j}\right\vert\right\vert+C\left\vert \left\vert \varphi_{j}\right\vert\right\vert. \label{Bu minus fj estimate}
\end{align}
manipulating equation \eqref{B of u nj estimate} and \eqref{Bu minus fj estimate} with $||\varphi_{j}||<\infty$ it follows that,
\begin{align}
    \left\vert \Lambda_{j}(u_{n})\right\vert_{\mathcal{H}} \leq CC_{1}||\varphi_{j}||+C||\varphi_{j}||:=C_{12}. \label{Bj of an estimate in H}
\end{align}
merging \eqref{Bj of an estimate in H} in \eqref{L5 before Bj of un estimate} concludes
\begin{align}
    \mathbb{E} \left(|L^{n}_{5}(s_{n}+a)-L^{n}_{5}(s_{n})|^2_{\mathcal{H}}\right)\leq  \mathbb{E}\left( \int_{s_{n}}^{s_{n}+a} C_{12}^2dt'\right)=C_{12}^2a. \label{L5 estimate}
\end{align}
Fix $\kappa, e >0$. Then using Chebyshev's inequality along with estimates \eqref{L2 estimate}, \eqref{L3 estimate}, \eqref{L4 estimate} and \eqref{L5 estimate}, results in 
\begin{align}
    \mathbb{P}(|L_{i}^n(s_n+a)-L_{i}^n(s_n)|_{\mathcal{H}} \geq \kappa) \leq \frac{1}{\kappa} \mathbb{E}\left[ |L_{i}^n(s_n+a)-L_{i}^n(s_n)|_{\mathcal{H}} \right] \leq \frac{c_ia}{\kappa},\nonumber
\end{align}
with $n \in \mathbb{N}$, $i=1,2,...,5$. Let $\delta_{i}=\frac{\kappa}{c_i}c$. Then 
\begin{align}
    \sup_{n \in N } \sup_{0\leq a \leq \delta_{i}} \mathbb{P}(|L_{i}^n(s_n+a)-L_{i}^n(s_n)|_{\mathcal{H}} \geq \kappa) \leq c, \nonumber
\end{align}
With the help of Chebyshev's inequality,
\begin{align}
    \mathbb{P}(|L_{5}^n(s_n+a)-L_{5}^n(s_n)|_{\mathcal{H}} \leq\frac{1}{\kappa^2}\mathbb{E}\left[ |L_{i}^n(s_n+a)-L_{i}^n(s_n)| \right] \kappa) \leq \frac{c_5a}{\kappa^2} \nonumber
\end{align}
Hence,
\begin{align}
    \sup_{n \in \mathbb{N} } \sup_{0\leq a \leq \delta_{i}} \mathbb{P}\left(\left|L_{5}^n(s_n+a)-L_{5}^n(s_n)\right|_{\mathcal{H}} \geq \kappa\right) \leq c, \nonumber
\end{align}
Since the Aldous condition holds for each of $L_{i}^n$, we conclude that it is also true for $u_n$. Thus, proof concluded.\end{proof}

\section{Convergence of Quadratic Variations} 
Before we apply the Martingale Representation Theorem, we require some convergence results, which we state and prove in this section.
With the help of Theorem \eqref{set of measures tight thm}. We have $(u_{n_{i}})_{i \in N}$ with probability measure space $(\hat{\Omega},\hat{F},
\hat{\mathbb{P}})$ in a manner that we have $\mathcal{Y}_{T}$ valued random variables $\hat{v}$ and $v_{n_{i}}, i\geq 1 $ with property that \\
\begin{align} \label{same law property for w and u}
 v_{n_{i}} \text{ and }  u_{n_{i}} \text{ have same distribution and }  v_{n_{i}} \text{ converge to } v \text{ in } \mathcal{Y}_{T} \quad\hat{\mathbb{P}}-\text{a.s}.   
\end{align}
\\ 
That means,
\begin{align*}
    &v_{n_{i}} \rightarrow  v \text{  in  } C\left([0,T],\mathcal{L}^{2}\right),\\
    &v_{n_{i}} \rightarrow  v \text{  in  } \mathcal{L}^2\left([0,T],D(\Delta)\right),\\
    &v_{n_{i}} \rightarrow  v \text{  in  } \mathcal{L}^2\left([0,T],\mathcal{V}\right),\\
    &v_{n_{i}} \rightarrow  v \text{  in  } C\left([0,T],\mathcal{V}_{w}\right).
\end{align*}
Denoting $v_{n_{i}}$ again as $v_{n}$.
Because 
\begin{itemize}
    \item $u_{n}$ belongs to $ \mathcal{C}([0,T], F_{n}) \quad \mathbb{P}-a.s.$
    \item The subset $ \mathcal{C}([0,T], F_{n}) \subset \mathcal{C}([0,T], \mathcal{H})\cup \mathcal{L}^2([0,T];\mathcal{V}) $ is Borel.
    \item $v_{n}$ and $u_{n}$ follow the same distribution.
\end{itemize}
Therefore, for all $n \geq 1$,
\begin{align*}
    \mathcal{L} (v_{n})(C ([0,T];F_n) &= 1, \;  \text{ i.e. }\\
    |v_{n}(t)|_{F_{n}}&=|u_{n}(t)|_{F_{n}}, \; \mathbb{P}-a.s.
\end{align*}
Using \eqref{same law property for w and u} we have $v_{n} \rightarrow v$ in $\mathcal{C}([0,T];\mathcal{L}^{2})$ and also from Theorem \eqref{Invarience in Finite dimension}, $u_{n}(t) \in \mathcal{M} $ and hence,
\begin{align}
    v_{n}(t) \in \mathcal{M}. \nonumber
\end{align}
With the help of Theorem \eqref{Important Estimates}, we have, 
\begin{align}
    \sup_{n\geq 1} \mathbf{\hat{\mathbb{E}}} \left(\sup_{s \in [0,T]}||v_{n}(t')||^{2}\right) \leq C_{1},\label{vn bound in V} \\
    \sup_{n\geq 1} \mathbf{\hat{\mathbb{E}}} \left(\sup_{s \in [0,T]}|v_{n}(t')|^{2n}_{{L^{2n}}}\right) \leq C_{2}, \label{vn bound in L2n} \\
    \sup_{n\geq 1}\mathbf{\hat{\mathbb{E}}}\left(\int_{0}^{T} |v_{n}(t')|^{2}_{\mathcal{E}}dt'\right) \leq C_{3}.\label{vn bound in Domain of A}
\end{align}

Keeping \eqref{vn bound in Domain of A} in the head, there exists a weakly convergent subsequence with the same label $v_{n}$ of $v_{n}$ in $\mathcal{L}^2([0,T]\times \hat{\Omega}; D(\Delta))$.
By \eqref{same law property for w and u} 
$ v_{n} \rightarrow v \text{ in } \mathcal{Y_{T}} ,\hat{\mathbb{P}}-$a.s, i.e.,
\begin{align}\label{v finite in E}
    \mathbf{\hat{\mathbb{E}}}\left(\int_{0}^{T} |v(t')|^{2}_{\mathcal{E}}dt'\right) < \infty.
\end{align}
Similarly, in the view of \eqref{vn bound in V} we have weak star convergent subsequence $v_{n}$ such that 
\begin{align}\label{fIniteness of v in V}
    \mathbf{\hat{\mathbb{E}}} \left(\sup_{0 \leq s \leq T}||v(t')||^{2}\right) < \infty. 
\end{align}
For each $n \geq 1$, we define a process $\hat{M}_{n}$
having paths in $C ([0,T];F_{n})$, and in
$C ([0,T];\mathcal{L}^{2})$ by

\begin{align}\label{hatM_n of v_n}
    \hat{M}_{n}&=v_{n}(t)-v_{n}(0) -\int_{0}^{t} \Delta v_{n}(t') dt'- \int_{0}^{t} F(v_{n}(t')) dt' -\frac{1}{2}\sum_{j=1}^{N} \int_{0}^{t} k_{j}(v_{n}(t')))dt'-\sum_{j=1}^{N} \int_{0}^{t} \Lambda_{j}(v_{n}(t')dW_{j}
\end{align}

\begin{lemma} \label{Quadratic Variation}
    Let $\left \langle\left \langle \hat{M_{n}} \right \rangle\right \rangle_t$ be the quadratic Variation of $\hat{M_{n}}$ then 
\begin{align}
    \left \langle\left \langle \hat{M_{n}} \right \rangle\right \rangle_t=\int_{0}^{t} (d\hat{M})^2=\sum_{j=1}^{N} \int_{0}^{t} |\Lambda_{j}(v_{n}(t'))|^{2} dt'. \nonumber
\end{align}
And $\hat{M_{n}}$ is a martingale which is also square integrable.
\end{lemma}
\begin{proof}
Because $u_n$ and $v_n$ have same laws, so for each $t_1,t_2 \in [0,T], t_1 \leq t_2$, for each $h$ from $\mathcal{C}([0,T],\mathcal{H})$, and for all $\phi_1, \phi_2 \in \mathcal{H}$, one obtains,
\begin{align}
    \hat{\mathbb{E}}\left[ \left \langle \hat{M}_{n}(t_1)-\hat{M}_{n}(t_2) , \phi_1 \right \rangle_{\mathcal{H}}h(v_{n|[0,t_2]}) \right]=0 \label{firt req for quad limit pass}
\end{align}
and
\begin{align}
    &\hat{\mathbb{E}} \left[ \left( \left \langle \hat{M}_{n}(t_1), \phi_1 \right \rangle_{\mathcal{H}}\left \langle \hat{M}_{n}(t_1), \phi_2 \right \rangle_{\mathcal{H}}- \left \langle \hat{M}_{n}(t_2), \phi_1 \right \rangle_{\mathcal{H}}\left \langle \hat{M}_{n}(t_2), \phi_2 \right \rangle_{\mathcal{H}}\right)h(v_{n|[0,t_2]})\right] \nonumber \\ &-\hat{\mathbb{E}}\left[\left( \sum_{j=1}^{m} \int_{t_2}^{t_1} \left \langle \Lambda_{j}(v_{n}(\tau)) \phi_1, \Lambda_{j}(v_{n}(\tau)) \phi_2 \right \rangle_{\mathcal{H}} d\tau \right) \cdot h(v_{n|[0,t_2]}). \right]=0. \label{2nd req for quad limit pass}
\end{align}
This ends the proof.
\end{proof}


\begin{lemma} \label{continuity of M hat}
    Consider the following process 
    \begin{align}
        \hat{M}(t)=v(t)-v(0)-\int_{0}^{t} \Delta v(t') dt'- \int_{0}^{t} F(v(t')) dt' -\frac{1}{2}\sum_{j=1}^{N} \int_{0}^{t} k_{j}(v(t')))dt'-\sum_{j=1}^{N} \int_{0}^{t} \Lambda_{j}(v(t')dW_{j}, \nonumber
    \end{align}
    the process $\hat{M}$ is a continuous process with values in $\mathcal{H}$. 
\end{lemma}
\begin{proof}
Keep in view that  $ v \in \mathcal{C}([0,T],\mathcal{V})$ therefore it is enough to establish the following,
\begin{align*}
    \left|\int_{0}^{t} \Delta v(t') dt'- \int_{0}^{t} F(v(t')) dt' -\frac{1}{2}\sum_{j=1}^{N} \int_{0}^{t} k_{j}(v(t')))dt'-\sum_{j=1}^{N} \int_{0}^{t} \Lambda_{j}(v(t')dW_{j}\right|_{\mathcal{L}^2} < \infty.
\end{align*}
For this purpose, we will deal with each term separately.\\
At first, we consider $\int_{0}^{t} | \Delta v(t') |_{\mathcal{L}^2} dt'$, with the help of Cauchy-Schwartz inequality we obtain
\begin{align*}
    \hat{\mathbb{E}}\left( \int_{0}^{t} | \Delta v(t') |_{\mathcal{L}^{2}} dt' \right)&\leq \hat{\mathbb{E}}\left( \int_{0}^{T} | \Delta v(t')|_{\mathcal{L}^{2}} dt' \right)\\
    & \leq  T^{\frac{1}{2}}\hat{\mathbb{E}}\left( \int_{0}^{T} | \Delta v(t')|^{2}_{\mathcal{L}^{2}} dt' \right)\\
    &\leq T^{\frac{1}{2}}\hat{\mathbb{E}}\left( \int_{0}^{T} | v(t')|^{2}_{\mathcal{E}} dt' \right)
\end{align*}
The above inequality in the view of \eqref{v finite in E} results in 
\begin{align}\label{v bound in E }
    \hat{\mathbb{E}}\left( \int_{0}^{t} | \Delta v(t') |_{\mathcal{L}^{2}} dt' \right) < \infty. 
\end{align}
Next, we see;
\begin{align*}
     \hat{\mathbb{E}}\left(\int_{0}^{t} \left|F(v(t'))\right|_{\mathcal{L}^2} dt'\right)&\leq  \hat{\mathbb{E}} \left( \int_{0}^{T} H_{1}\left(\left\vert \right\vert v(t') \left\vert \right\vert ,0 \right)  \left\vert\right\vert v(t') \left \vert \right\vert dt'\right)\\
     & \leq \hat{\mathbb{E}} \left( \int_{0}^{T} C_{10}  \left\vert\right\vert v(t') \left \vert \right\vert dt'  \right)\\
     &\leq C_{10} T^{\frac{1}{2}} \hat{\mathbb{E}} \left( \int_{0}^{T} \left\vert\right\vert v(t') \left \vert \right\vert^{2} dt'  \right)^{\frac{1}{2}}.
\end{align*}
Using \eqref{fIniteness of v in V}, we conclude 

\begin{align}\label{F(v) bound in L2}
    \hat{\mathbb{E}}\left(\int_{0}^{t} \left|F(v(t'))\right|_{\mathcal{L}^2} dt'\right) < \infty.
\end{align}
Then we also have the following,
\begin{align}\label{kj(v) bound in L2}
    \hat{\mathbb{E}}\left(\int_{0}^{t} |k_{j}(v(t')))| dt'\right) \leq  \hat{\mathbb{E}} \left( \int_{0}^{T} C_{11} dt' \right) .=C_{11}T<\infty
\end{align}
Finally, because $\int_{0}^{t} \Lambda_{j}(v(t'))dW_{j}$ is stochastic integral so
\begin{align}\label{Bj(v) bound}
    \hat{\mathbb{E}}\left(\int_{0}^{t} \Lambda_{j}(v(t'))dW_{j}\right)=0.
\end{align}
Combining \eqref{v bound in E }, \eqref{F(v) bound in L2}, \eqref{kj(v) bound in L2} and \eqref{Bj(v) bound} we get 

\begin{align*}
    \left|\int_{0}^{t} \Delta v(t') dt'- \int_{0}^{t} F(v(t')) dt' -\frac{1}{2}\sum_{j=1}^{N} \int_{0}^{t} k_{j}(v(t')))dt'-\sum_{j=1}^{N} \int_{0}^{t} \Lambda_{j}(v(t')dW_{j}\right|_{\mathcal{L}^2} < \infty.
\end{align*}
Proof gets completed.
\end{proof}
\begin{theorem} \label{convergence in Ehat}
    For $s,t \in [0,T]$ with $s \leq t$ and $\phi \in \mathcal{H}$, the following are true:
\begin{align*}
&1. \;\lim_{n \rightarrow \infty } \left \langle v_{n}(t), Q_{n}(\phi) \right \rangle_{\mathcal{H}}=\left \langle v(t) , \phi \right \rangle_{\mathcal{H}},\; \hat{\mathbb{P}}-a.s. \,. \\
&2.\; \lim_{n \rightarrow \infty } \int_{s}^{t} \left  \langle A v_{n}(h), Q_{n}(\phi)  \right \rangle_{\mathcal{H}} dh=\int_{s}^{t} \left  \langle A v(h),\phi  \right \rangle_{\mathcal{H}},\; \hat{\mathbb{P}}-a.s.\, .\\
&3.\; \lim_{n \rightarrow \infty } \int_{s}^{t} \left \langle  F(v_{n}(h)) , Q_{n}(\phi) \right \rangle_{\mathcal{H}} dh = \int_{s}^{t} \left \langle F(v(h)), Q(\phi) \right \rangle_{\mathcal{H}} dh, \; \; \hat{\mathbb{P}}-a.s.\; \text{ for }\phi \in \mathcal{H}.  \\
&4.\; \lim_{n \rightarrow \infty } \int_{s}^{t} \left \langle  k_{j}(v_{n}(h)) , Q_{n}(\phi) \right \rangle_{\mathcal{H}} dh = \int_{s}^{t} \left \langle k_{j}(v(h)), Q(\phi) \right \rangle_{\mathcal{H}} dh,\; \; \hat{\mathbb{P}}-a.s.\, . \\
&5.\; \lim_{n \rightarrow \infty } \int_{s}^{t} \left \langle  \Lambda_{j}(v_{n}(h)) , Q_{n}(\phi) \right \rangle_{\mathcal{H}} dW_{j}(h) = \int_{s}^{t} \left \langle \Lambda_{j}(v(h)), \phi \right \rangle_{\mathcal{H}} dW_{j}(h) ,\; \; \hat{\mathbb{P}}-a.s.\, .
\end{align*}
\end{theorem}
\begin{proof}
By \eqref{same law property for w and u}, we have 
\begin{align}
    v_{n} \rightarrow v \text{ in } \mathcal{Y}_T .\nonumber
\end{align}
Since 
\begin{equation*}
    v_{n} \rightarrow v \text{ in } \mathcal{C}([0,T], \mathcal{H}), \; \hat{\mathbb{P}}-a.s
\end{equation*}
and 
\begin{equation*}
    Q_{n}(\phi) \rightarrow \phi \; \text{ in } \mathcal{H}. 
\end{equation*}
we get ;
\begin{align*}
    \lim_{n \rightarrow \infty } \left \langle v_{n}(t) , Q_{n} (\phi) \right \rangle - \left \langle v(t), \phi \right \rangle= \lim_{n \rightarrow \infty } \left \langle v_{n}(t)-v(t) , Q_{n} (\phi) \right \rangle+ \lim_{n \rightarrow \infty } \left \langle v(t) , Q_{n} (\phi)-\phi \right \rangle=0, \; \hat{\mathbb{P}}-a.s
\end{align*}
This proves inference (1).\\
Pick $\phi \in \mathcal{H}$
\begin{align*}
    &\int_{s}^{t} \left \langle A(v_{n}(h)) , Q_{n} (\phi) \right \rangle_{\mathcal{H}} dh - \int_{s}^{t} \left \langle A(v(h)) , \phi \right \rangle_{\mathcal{H}} dh\\ 
    &=   \int_{s}^{t}  \left \langle A(v_{n}(h))-A(v(h)) , \phi \right \rangle_{\mathcal{H}} dh+   \int_{s}^{t}  \left \langle A(v_{n}(h)) ,Q_{n}(\phi) - \phi \right \rangle_{\mathcal{H}} dh\\
    &=\int_{s}^{t}  \left \langle A(v_{n}(h)-v(h)) , \phi \right \rangle_{\mathcal{H}} dh+   \int_{s}^{t}  \left \langle A(v_{n}(h)) ,Q_{n}(\phi) - \phi \right \rangle_{\mathcal{H}} dh\\
    &\leq \int_{s}^{t}  \left \langle v_{n}(h)-v(h), A^{-1}\phi \right \rangle_{D(A)} dh+  \int_{s}^{t}  \left \vert A(v_{n}(h)) \right \vert_{\mathcal{H}} |Q_{n}(\phi) - \phi |_{\mathcal{H}} dh\\
    &\leq \int_{s}^{t}  \left \langle v_{n}(h)-v(h), A^{-1}\phi \right \rangle_{D(A)} dh+  \int_{0}^{T}  \left \vert v_{n}(h) \right \vert_{D(A)} |Q_{n}(\phi) - \phi |_{\mathcal{H}} dh\\
    &\leq \int_{s}^{t}  \left \langle v_{n}(h)-v(h), A^{-1}\phi \right \rangle_{D(A)} dh+  \left( \int_{0}^{T}  \left \vert v_{n}(h) \right \vert_{D(A)}^{2} dh\right)^{\frac{1}{2}} \left(\int_{0}^{T} |Q_{n}(\phi) - \phi |_{\mathcal{H}}^{2} dh \right )^{\frac{1}{2}}\\
    &=\int_{s}^{t}  \left \langle v_{n}(h)-v(h), A^{-1}\phi \right \rangle_{D(A)} dh+  \left( \int_{0}^{T}  \left \vert v_{n}(h) \right \vert_{D(A)}^{2} dh\right)^{\frac{1}{2}} \left(  |Q_{n}(\phi) - \phi |_{\mathcal{H}}^{2} \int_{0}^{T} dh \right )^{\frac{1}{2}}\\
    &=\int_{s}^{t}  \left \langle v_{n}(h)-v(h), A^{-1}\phi \right \rangle_{D(A)} dh+  |v_{n}|_{\mathcal{L}^{2}([0,T], D(A))}) \left(  |Q_{n}(\phi) - \phi |_{\mathcal{H}}^{2} T \right )^{\frac{1}{2}}\\
    &=\int_{s}^{t}  \left \langle v_{n}(h)-v(h), A^{-1}\phi \right \rangle_{D(A)} dh+  |v_{n}|_{\mathcal{L}^{2}([0,T], D(A))}  |Q_{n}(\phi) - \phi |_{\mathcal{H}} T^{\frac{1}{2}}.\\
\end{align*}
Since $v_{n}$, $\hat{\mathbb{P}}$-a.s. converges weakly to limit $v$ in $\mathcal{L}^{2}([0,T], D(A))$,  also $v_{n}$  is uniformly bounded in $\mathcal{L}^{2}([0,T], D(A))$ and $Q_{n}(\phi)$ converges to$ \phi$ in $\mathcal{H}$-norm. This gives
\begin{align*}
    \lim_{n \rightarrow \infty } \int_{s}^{t} \int_{s}^{t}  \left \langle v_{n}(h)-v(h), A^{-1}\phi \right \rangle_{D(A)} dh=0,
\end{align*}
and 
\begin{align*}
    \lim_{n \rightarrow \infty} |Q_{n}(\phi)-\phi|_{\mathcal{H}}=0,
\end{align*}
giving 
\begin{align*}
    \lim_{n \rightarrow \infty }\int_{s}^{t} \left \langle A(v_{n}(h)) , Q_{n} (\phi) \right \rangle_{\mathcal{H}} dh = \int_{s}^{t} \left \langle A(v(h)) , \phi \right \rangle_{\mathcal{H}} dh,
\end{align*}
and (2) gets proved.\\

For (3), suppose $\phi \in \mathcal{V}$
\begin{align*}
    &\int_{s}^{t} \left \langle F(v_{n}(h)), Q_{n}(\phi) \right \rangle_{\mathcal{H}} dh-\int_{s}^{t} \left \langle F(v(h)), \phi \right \rangle_{\mathcal{H}} dh,\\
    &= \int_{s}^{t} \left \langle F(v_{n}(h))-F(v(h)), \phi \right \rangle_{\mathcal{H}} dh+\int_{s}^{t} \left \langle F(v_{n}(h)), Q_{n}(\phi)-\phi \right \rangle_{\mathcal{H}} dh,\\
    &\leq \int_{0}^{T} | F(v_{n}(h))-F(v(h))|_{\mathcal{H}}|\phi|_{\mathcal{H}} dh+\int_{0}^{T} | F(v_{n}(h))|_{\mathcal{H}} |Q_{n}(\phi)-\phi|_{\mathcal{H}} dh.
\end{align*}

From Lipschitz property $H$ and $\mathcal{V} \hookrightarrow \mathcal{H}$ we have,
\begin{align*}
    &\int_{s}^{t} \left \langle F(v_{n}(h)), Q_{n}(\phi) \right \rangle_{\mathcal{H}} dh-\int_{s}^{t} \left \langle F(v(h)), \phi \right \rangle_{\mathcal{H}} dh\\
    &\leq \int_{0}^{T} \left ( H_{1}(||v_{n}(h)||,||v(h)||) ||v_{n}(h)-v(h)||\right)  |\phi|_{\mathcal{H}} dh+\int_{0}^{T}  | F(v_{n}(h))|_{\mathcal{H}} ||Q_{n}(\phi)-\phi|| dh\\
    &\leq \left( \int_{0}^{T} \left ( H_{1}(||v_{n}(h)||,||v(h)||)\right)^{2}dh \right)^{\frac{1}{2}}\cdot  \left( \int_{0}^{T} ||v_{n}(h)-v(h)||^{2}dh \right)^{\frac{1}{2}}\\&+\left(\int_{0}^{T}  | F(v_{n}(h))|_{\mathcal{H}}^{2} dh \right)^{\frac{1}{2}}\cdot \left(\int_{0}^{T} ||Q_{n}(\phi)-\phi||^{2} dh \right)^{\frac{1}{2}} 
\end{align*}
From $v_{n}(h) \rightarrow v(h) $ strongly in $\mathcal{L}^{2}([0,T], \mathcal{V})$, uniform boundedness of $\mathcal{L}^{2}([0,T],\mathcal{V})$  and $Q_{n}(\phi) \rightarrow \phi $ in $\mathcal{V}$ we get
\begin{align*}
    \lim_{n \rightarrow \infty }  &\int_{s}^{t} \left \langle F(v_{n}(h)), Q_{n}(\phi) \right \rangle_{\mathcal{H}} dh-\int_{s}^{t} \left \langle F(v(h)), \phi \right \rangle_{\mathcal{H}} dh=0, \; \; \hat{\mathbb{P}}-a.s.   
\end{align*}
which ends the proof for (3).\\
Let's prove (4), take $\phi \in \mathcal{H}$,
\begin{align*}
    &\int_{s}^{t} \left \langle  k_{j}(v_{n}(h)) , Q_{n}(\phi) \right \rangle_{\mathcal{H}} dh - \int_{s}^{t} \left \langle     k_{j}(v(h)), Q(\phi) \right \rangle_{\mathcal{H}} dh\\
    &=\int_{s}^{t} \left \langle  k_{j}(v_{n}(h))-k_{j}(v(h)) , \phi \right \rangle_{\mathcal{H}} dh - \int_{s}^{t} \left \langle     k_{j}(v_{n}(h)), Q_{n}(\phi)+\phi \right \rangle_{\mathcal{H}} dh,\\
    &\leq \int_{0}^{T} \left \vert k_{j}(v_{n}(h))-k_{j}(v(h))\right \vert_{\mathcal{H}} \left \vert \phi \right\vert_{\mathcal{H}} dh + \int_{0}^{T} \left \vert k_{j}(v_{n}(h)\right \vert_{\mathcal{H}} \left \vert Q_{n}(\phi)-\phi\right \vert_{\mathcal{H}} dh. 
\end{align*}
Using Lipchitz property of $k_{j}(v_{n}(h))$ along with $\mathcal{V} \hookrightarrow \mathcal{H}$ and \eqref{kj estimate} we obtain
\begin{align*}
    &\int_{s}^{t} \left \langle  k_{j}(v_{n}(h)) , Q_{n}(\phi) \right \rangle_{\mathcal{H}} dh - \int_{s}^{t} \left \langle     k_{j}(v(h)), Q(\phi) \right \rangle_{\mathcal{H}} dh\\
    &\leq \int_{0}^{T} C ||\varphi_{j}||^{2} \left[ 2+ ||v_{n}(h)||^{2}+||v(h)||^{2} +\left(||v_{n}(h)||+||v(h)||\right)^{2}\right ]||v_{n}(h)-v(h)|| \left \vert \phi \right\vert_{\mathcal{H}} dh \\ &+\int_{0}^{T} \left \vert k_{j}(v_{n}(h)\right \vert_{\mathcal{H}} C \left \vert \right \vert Q_{n}(\phi)-\phi \left\vert\right \vert dh, \\
    &\leq \left[ \int_{0}^{T} \left(C ||\varphi_{j}||^{2} \left[ 2+ ||v_{n}(h)||^{2}+||v(h)||^{2} +\left(||v_{n}(h)||+||v(h)||\right)^{2}\right ]\right)^{2} \left \vert \phi \right\vert_{\mathcal{H}}^{2} dh  \right]^{\frac{1}{2}} \\ &\left[ \int_{0}^{T} ||v_{n}(h)-v(h)||^{2} dh\right]^{\frac{1}{2}}
    \left[\int_{0}^{T} \left\vert k_{j}(v_{n}(h))\right\vert_{\mathcal{H}}^{2}\right]^{\frac{1}{2}} \cdot  \left[ \int_{0}^{T} \left \vert \right \vert Q_{n}(\phi)-\phi \left\vert\right \vert^{2}dh\right ]^{\frac{1}{2}} 
\end{align*}
By similar argumentation as used in (3),

\begin{align*}
    &\left[ \int_{0}^{T} \left(C ||\varphi_{j}||^{2} \left[ 2+ ||v_{n}(h)||^{2}+||v(h)||^{2} +\left(||v_{n}(h)||+||v(h)||\right)^{2}\right ]\right)^{2} \left \vert \phi \right\vert_{\mathcal{H}}^{2} dh  \right]^{\frac{1}{2}} \\& \left[ \int_{0}^{T} ||v_{n}(h)-v(h)||^{2} dh\right]^{\frac{1}{2}}
    \left[\int_{0}^{T} \left\vert k_{j}(v_{n}(h))\right\vert_{\mathcal{H}}^{2}\right]^{\frac{1}{2}} \cdot  \left[ \int_{0}^{T} \left \vert \right \vert Q_{n}(\phi)-\phi \left\vert\right \vert^{2}dh\right ]^{\frac{1}{2}} \rightarrow 0 
\end{align*}
which shows (4) also hold.

Finally, we will show that (5) also holds; for this, let  $\phi \in \mathcal{H}$ along with following Lipschitz condition \cite{brzezniak2020global}, 
\begin{equation*}
    \|\Lambda_{j}(v_{n}(h)-\Lambda_{j}(v(h)))\|_{\mathcal{H}} \leq |\varphi_{j}|_{\mathcal{H}}\left(|v_{n}(h)|_{\mathcal{H}}+|v(h)|_{\mathcal{H}}\right)\left\vert v_{n}(h)-v(h)\right\vert_{\mathcal{H}}. 
\end{equation*}    

\begin{align*}
  &\int_{s}^{t} \left \langle  \Lambda_{j}(v_{n}(h)) , Q_{n}(\phi) \right \rangle_{\mathcal{H}} dW_{j}(h) -\int_{s}^{t} \left \langle \Lambda_{j}(v(h)), Q(\phi) \right \rangle_{\mathcal{H}} dW_{j}(h)\\
  &= \int_{s}^{t} \left \langle  \Lambda_{j}(v_{n}(h))-\Lambda_{j}(v(h)) , \phi \right \rangle_{\mathcal{H}} dW_{j}(h) +\int_{s}^{t} \left \langle \Lambda_{j}(v_{n}(h)), Q_{n}(\phi)-\phi \right \rangle_{\mathcal{H}} dW_{j}(h)\\
  & \leq \int_{0}^{T} | \Lambda_{j}(v_{n}(h))-\Lambda_{j}(v(h)) |_{\mathcal{H}} | \phi |_{\mathcal{H}} dW_{j}(h) +\int_{0}^{T} | \Lambda_{j}(v_{n}(h)) |_{\mathcal{H}} | Q_{n}(\phi)-\phi |_{\mathcal{H}} dW_{j}(h)\\
  &\leq \int_{0}^{T}  |\varphi_{j}|_{\mathcal{H}}\left(|v_{n}(h)|_{\mathcal{H}}+|v(h)|_{\mathcal{H}}\right)\left\vert v_{n}(h)-v(h)\right\vert_{\mathcal{H}} | \phi |_{\mathcal{H}} dW_{j}(h) +\int_{0}^{T} | \Lambda_{j}(v_{n}(h)) |_{\mathcal{H}} | Q_{n}(\phi)-\phi |_{\mathcal{H}} dW_{j}(h)\\
  &=A_{1}+A_{2}
\end{align*}
Then, by Ito Isometry,
\begin{align*}
    &\hat{\mathbb{E}}(A_{1}^{2})=\hat{\mathbb{E}}\left(\int_{0}^{T} | \Lambda_{j}(v_{n}(h))-\Lambda_{j}(v(h)) |_{\mathcal{H}} | \phi |_{\mathcal{H}} dW_{j}(h)\right)^{2}\\
    &=\hat{\mathbb{E}} \left(\int_{0}^{T}   |\varphi_{j}|_{\mathcal{H}}^{2}\left(|v_{n}(h)|_{\mathcal{H}}+|v_{h}|_{\mathcal{H}}\right)^{2}\left\vert v_{n}(h)-v(h)\right\vert_{\mathcal{H}}^{2} | \phi |_{\mathcal{H}}^{2}  d(h)\right)\\
    &\leq \hat{\mathbb{E}} \left[\left(\int_{0}^{T}  |\varphi_{j}|_{\mathcal{H}}^{4}\left(|v_{n}(h)|_{\mathcal{H}}+|v_{h}|_{\mathcal{H}}\right)^{4}| \phi |_{\mathcal{H}}^{4}  dh\right)^{\frac{1}{2}} \cdot \left(\int_{0}^{T}  \left\vert v_{n}(h)-v(h)\right\vert_{\mathcal{H}}^{4} dh\right)^{\frac{1}{2}} \right]\\
    &\leq \hat{\mathbb{E}} \left[\left(\int_{0}^{T}  |\varphi_{j}|_{\mathcal{H}}^{4}\left(|v_{n}(h)|_{\mathcal{H}}+|v_{h}|_{\mathcal{H}}\right)^{4}| \phi |_{\mathcal{H}}^{4}  dh\right)^{\frac{1}{2}} \cdot |v_{n}(h)-v(h)|_{\mathcal{L}^4([0,T],\mathcal{H})}^2 \right].
\end{align*}
Using the continuity of embedding $\mathcal{C}([0,T], \mathcal{H}) \hookrightarrow {\mathcal{L}^4([0,T],\mathcal{H})}$, it follows that,
\begin{align*}
    &\hat{\mathbb{E}}(A_{1}^{2})=\hat{\mathbb{E}}\left(\int_{0}^{T} | \Lambda_{j}(v_{n}(h))-\Lambda_{j}(v(h)) |_{\mathcal{H}} | \phi |_{\mathcal{H}} dW_{j}(h)\right)^{2}\\
    &\leq \hat{\mathbb{E}} \left[\left(\int_{0}^{T}  |\varphi_{j}|_{\mathcal{H}}^{4}\left(|v_{n}(h)|_{\mathcal{H}}+|v_{h}|_{\mathcal{H}}\right)^{4}| \phi |_{\mathcal{H}}^{4}  dh\right)^{\frac{1}{2}} \cdot C^{2}|v_{n}(h)-v(h)|_{\mathcal{C}([0,T],\mathcal{H})}^2 \right]
\end{align*}
since $v_{n}(h) \rightarrow v(h)$ in $\mathcal{C}([0,T],\mathcal{H})$ so

\begin{align*}
    \lim_{n \rightarrow \infty } \hat{\mathbb{E}}(A_{1}^{2})=\hat{\mathbb{E}}\left(\int_{0}^{T} | \Lambda_{j}(v_{n}(h))-\Lambda_{j}(v(h)) |_{\mathcal{H}} | \phi |_{\mathcal{H}} dW_{j}(h)\right)^{2}=0.
\end{align*}

Now we see $\hat{\mathbb{E}}\left(A_{2}^{2}\right)$;

\begin{align*}
    \hat{\mathbb{E}}\left(A_{2}^{2}\right)&=\hat{\mathbb{E}}\left( \int_{0}^{T} c_{2} | \Lambda_{j}(v_{n}(h)) |_{\mathcal{H}} | Q_{n}(\phi)-\phi |_{\mathcal{H}} dW_{j}(h)\right)^{2}\\
    &=\hat{\mathbb{E}}\left( \int_{0}^{T} c_{2} | \Lambda_{j}(v_{n}(h)) |_{\mathcal{H}}^{2} | Q_{n}(\phi)-\phi |_{\mathcal{H}}^{2} dW_{j}(h)\right)\\
    &=\hat{\mathbb{E}}\left(A_{2}^{2}\right)=\hat{\mathbb{E}}\left(  | Q_{n}(\phi)-\phi |_{\mathcal{H}}^{2} \int_{0}^{T} c_{2} | \Lambda_{j}(v_{n}(h)) |_{\mathcal{H}}^{2} dW_{j}(h)\right)
\end{align*}
Since $ Q_{n}\phi \rightarrow \phi \text{ in } \mathcal{H} $ so $ |Q_{n}\phi-\phi|_{\mathcal{H}} \rightarrow 0$ giving 
\begin{align*}
    \lim_{n \rightarrow \infty} \hat{\mathbb{E}}\left(A_{2}^{2}\right)=\hat{\mathbb{E}}\left(  | Q_{n}(\phi)-\phi |_{\mathcal{H}}^{2} \int_{0}^{T} c_{2} | \Lambda_{j}(v_{n}(h)) |_{\mathcal{H}}^{2} dW_{j}(h)\right)=0
\end{align*}

Therefore,
\begin{align*}
    \hat{\mathbb{E}} \int_{s}^{t} \left \langle  \Lambda_{j}(v_{n}(h)) , Q_{n}(\phi) \right \rangle_{\mathcal{H}} dW_{j}(h) -\int_{s}^{t} \left \langle \Lambda_{j}(v(h)), Q(\phi) \right \rangle_{\mathcal{H}} dW_{j}(h) \leq \hat{\mathbb{E}}(A_{1}^{2})+\hat{\mathbb{E}}(A_{2}^{2}) \rightarrow 0, 
\end{align*}
as $ n \rightarrow \infty$. Thus, (5) holds, and theorem \eqref{convergence in Ehat} is proved.
\end{proof}

\begin{lemma}\label{mns1-mns2 lemma used for passing limit}
    For all $t_1,t_2 \in [0,T]$ with $t_2 \leq t_1$, for all $\phi$ in $ \mathcal{V}$ and for all $h$ in $\mathcal{C}([0,T], \mathcal{H})$. Then following holds.
    \begin{align}
        \lim_{n \rightarrow \infty }\hat{\mathbb{E}}\left[ \left \langle \hat{M}_{n}(t_1)-\hat{M}_{n}(t_2) , \phi_1 \right \rangle_{\mathcal{H}}h(v_{n|[0,t_2]}) \right]= \hat{\mathbb{E}}\left[ \left \langle \hat{M}(t_1)-\hat{M}(t_2) , \phi_1 \right \rangle_{\mathcal{H}}h(v_{|[0,t_2]}) \right]. \label{ms1-ms2 conv lemma}
    \end{align}
\end{lemma}
\begin{proof}
Let $t_1,t_2 \in [0,T]$ with $t_2 \leq t_1$ and $\phi \in \mathcal{V} $then using \eqref{hatM_n of v_n}  we get 
\begin{align}
    &\left \langle \hat{M}_{n}(t_1)-\hat{M}_{n}(t_2),\phi \right \rangle \nonumber \\ 
    &=\left \langle v_{n}(t_1)-v_{n}(t_2), Q_{n}\phi \right \rangle - \int_{t_2}^{t_1} \left \langle \Delta (v_{n}(\tau)),Q_{n} \phi \right \rangle_{\mathcal{H}} d\tau - \int_{t_2}^{t_1} \left \langle F(v_n(\tau)), Q_{n}\phi \right \rangle_{\mathcal{H}} d\tau \nonumber \\  &-\frac{1}{2} \sum_{1}^{N} \int_{t_2}^{t_1} \left \langle k_{j}(v_n(\tau)), Q_{n} \phi \right \rangle_{\mathcal{H}} d\tau - \sum_{j=1}^{N} \int_{t_2}^{t_1} \left \langle \Lambda_{j}(v_{n}(\tau)) , Q_{n}\phi \right \rangle d\tau \label{Ms1-Ms2 conv}
\end{align}
Applying Theorem \eqref{convergence in Ehat} in \eqref{Ms1-Ms2 conv} we have 
\begin{align} 
    \lim_{n \rightarrow \infty} \left \langle \hat{M}_{n}(t_1)-\hat{M}_{n}(t_2),\phi \right \rangle=\left \langle \hat{M}(t_1)-\hat{M}(t_2),\phi \right \rangle \; \hat{\mathbb{P}}-a.s.  \label{Mns1-Mns2 converges}
\end{align}
Now we will prove \eqref{ms1-ms2 conv lemma}. For this, we already have $v_{n} \rightarrow v$ in $\mathcal{Y}_{T}$ this way $v_{n} \rightarrow v$ in $\mathcal{C}([0,T],\mathcal{H})$, which yields, 
\begin{align}
    \lim_{n \rightarrow \infty } h(v_{n|[0,t_2]]})=h(v_{|[0,t_2]})\; \; \; \hat{\mathbb{P}}-a.s. \label{hvn converge to hv }
\end{align}
and 
\begin{align}
    \sup_{ n \in N } |h(v_{n|[0,t_2]})|_{\mathcal{L}^{\infty}} < \infty.  \nonumber
\end{align}
Let's set, 
\begin{align}
    \phi_{n}(\omega):= \left[  \left \langle \hat{M}_{n}(t_1,\omega)-\hat{M}_{n}(t_2,\omega) , \phi_1 \right \rangle_{\mathcal{H}}h(v_{n|[0,t_2]}) \right], \; w \in \hat{\Omega} \nonumber
\end{align}
we will show that 
\begin{align}
    \sup_{n \geq 1} \hat{\mathbb{E}} (|\phi_{n}|^2) < \infty \label{claim of phi n conv in E hat}
\end{align}
With the help of Cauchy-Schwarz inequality and $\mathcal{V} \hookrightarrow \mathcal{V}'$ $\forall n \in \mathbb{N}$ we have a positive constant $c$ such that 
\begin{align}
    \hat{\mathbb{E}}\left[ |\phi_{n}|^2 \right] \leq c |h|_{\mathcal{L}^\infty}^2 |\phi_1|_{\mathcal{V}}^{2} \hat{\mathbb{E}}\left[ |\hat{M}_{n}(t_1)|^{2}_{\mathcal{H}}+|\hat{M}_{n}(t_2)|^{2}_{\mathcal{H}}\right] \nonumber,  
\end{align}
because $\hat{M}_{n}$ is a continuous process and martingale having quadratic variation as described in Lemma \eqref{Quadratic Variation}, using Burkholder-Inequality we have 
\begin{align}\label{Mn hat in E hat }
\hat{\mathbb{E}}[\sup_{s \in [0,T]} |\hat{M}_{n}(t')|^{2}_{\mathcal{H}}] \leq c \hat{\mathbb{E}}\left( \sum_{j=1}^{N} \int_{0}^{T} |\Lambda_{j}(v_{n}(\tau))|^{2}_{\mathcal{H}} d\tau\right)    
\end{align}
equation \eqref{Mn hat in E hat } along with \eqref{Bj of an estimate in H} implies
\begin{align}
    &\hat{\mathbb{E}}[\sup_{s \in [0,T]} |\hat{M}_{n}(t')|^{2}_{\mathcal{H}}] \leq \hat{\mathbb{E}} \left( \sum_{j=1}^{N} \int_{0}^{T} C_{12}^{2} d\tau\right) \nonumber \\
    &\hat{\mathbb{E}}[\sup_{s \in [0,T]} |\hat{M}_{n}(t')|^{2}_{\mathcal{H}}] \leq \hat{\mathbb{E}}\left( C_{12}^{2}TN \right) =C_{12}^{2}TN < \infty. \label{sup Mn square norm in E hat}
\end{align}

Then, \eqref{Mn hat in E hat } together with \eqref{sup Mn square norm in E hat} allows us to conclude that \eqref{claim of phi n conv in E hat} holds. By using the fact of uniform integrability of the sequence $\phi_n$ is and employing \eqref{Mns1-Mns2 converges}, we infer that $\phi_n$ converges $\mathbb{P}-a.s.$. Thus, the application of the Vitali Theorem gives the required proof. 
\end{proof}
\begin{corollary}\label{corollary used in proving main theorem}
    For all $t_1,t_2 \in [0,T]$ such that $t_2 \leq t_1$ one has 
    \begin{align}
        \hat{\mathbb{E}}\left( \hat{M}(t_1)-\hat{M}(t_2)| \mathcal{F}_{t_1} \right) =0. \nonumber
    \end{align}
\end{corollary}

\begin{lemma}\label{2nd lemma for passing limit}
    For all $t_1,t_2 \in [0,T]$ with $t_2 \leq t_1$ and $\forall \phi_1, \phi_2 \in \mathcal{V}$ we have 
    \begin{align}
        &\lim_{n \rightarrow \infty } \hat{\mathbb{E}} \left[  \left(\left \langle \hat{M}_{n}(t_1),\phi_1 \right \rangle \left \langle \hat{M}_{n}(t_1),\phi_2 \right \rangle-\left \langle \hat{M}_{n}(t_2),\phi_1 \right \rangle \left \langle \hat{M}_{n}(t_2),\phi_2 \right \rangle \right) h(v_{n|[0,t_2]})\right] \nonumber \\&=\hat{\mathbb{E}} \left[ \left(\left \langle \hat{M}(t_1),\phi_1 \right \rangle \left \langle \hat{M}(t_1),\phi_2 \right \rangle-\left \langle \hat{M}(t_2),\phi_1 \right \rangle \left \langle \hat{M}(t_2),\phi_2 \right \rangle \right) h(v_{n|[0,t_2]}) \right], \; \; h \in  \mathcal{C}([0,T],\mathcal{H}), \nonumber
    \end{align}
    where $\left \langle \cdot , \cdot \right \rangle$ represents duality product in the spaces $\mathcal{V}$ and $\mathcal{V}'$.
\end{lemma}
\begin{proof}
Let $t_1,t_2 \in [0,T]$ with $t_2 \leq t_1$ and $\forall \phi_1, \phi_2 \in \mathcal{V}$.\\
Set:\\ 
$\phi_{n}(\omega):= \left(\left \langle \hat{M}_{n}(t_1,\omega),\phi_1 \right \rangle \left \langle \hat{M}_{n}(t_1,\omega),\phi_2 \right \rangle-\left \langle \hat{M}_{n}(t_2,\omega),\phi_1 \right \rangle \left \langle \hat{M}_{n}(t_2,\omega),\phi_2 \right \rangle \right) h(v_{n|[0,t_2]}(\omega)) $,\\
$\phi(\omega):= \left(\left \langle \hat{M}(t_1,\omega),\phi_1 \right \rangle \left \langle \hat{M}(t_1,\omega),\phi_2 \right \rangle-\left \langle \hat{M}(t_2,\omega),\phi_1 \right \rangle \left \langle \hat{M}(t_2,\omega),\phi_2 \right \rangle \right) h(v_{|[0,t_2]}(\omega)), \;\; \omega \in \hat{\Omega}$

From \eqref{Mns1-Mns2 converges} and \eqref{hvn converge to hv } it is concluded that 
\begin{align}
    \lim_{n \rightarrow \infty } \phi_{n}(\omega)=\phi(\omega), \;\; \hat{\mathbb{P}}-a.s \nonumber
\end{align} 
We claim that there is $r>1$ we have 
\begin{align}
    \sup_{n \geq 1}\hat{\mathbb{E}}\left[ |\phi_{n}|^r \right] < \infty  \label{claim for phi n in Ehat}
\end{align}

We have the following inequality similar to the one we had in the previous lemma.

\begin{align}
    \hat{\mathbb{E}}\left[ |\phi_{n}|^r \right] \leq c |h|_{\mathcal{L}^{\infty}}^{r}|\phi_1|_{\mathcal{V}}^r|\phi_2|_{\mathcal{V}}^r \hat{\mathbb{E}}\left[ |\hat{M}_{n}(t_1)|^{2r}+|\hat{M}_{n}(t_2)|^{2r} \right]. \label{phi n power r in hat E}
\end{align}
Again, by Burkholder inequality Theorem, we have 
\begin{align}
    \hat{\mathbb{E}}\left[ \sup_{s \in [0,T]} |\hat{M}_{n}(t')|^{2r} \right] \leq c \hat{\mathbb{E}}\left[ \sum_{j=1}^{N} \int_{0}^{T} |\Lambda_{j}(v_{n}(\tau))|^2 d\tau \right]^r \label{ Mn 2r in E hat}
\end{align}
Also
\begin{align}
 \hat{\mathbb{E}}\left[ \sum_{j=1}^{N} \int_{0}^{T} |\Lambda_{j}(v_{n}(\tau))|^2 d\tau \right]^r \leq \hat{\mathbb{E}}\left( \sum_{j=1}^{N} C_{12}^2 T\right)^r=N^rC_{12}^{2r}T^r < \infty.  \label{Bjn in E hat}
\end{align}
From \eqref{phi n power r in hat E}, \eqref{ Mn 2r in E hat} and \eqref{Bjn in E hat} makes \eqref{claim for phi n in Ehat} true. The application of the Vitali Theorem  yields  
\begin{align}
    \lim_{n \rightarrow \infty } \hat{\mathbb{E}}(\phi_{n})=\hat{\mathbb{E}}(\phi). \nonumber
\end{align}
\end{proof}

\subsection{\textbf{Convergence of Quadratic Variations}}

\begin{lemma}\label{3rd for passing limit in quad variations}
    $\forall t_1,t_2 \in [0,T], \forall \phi_1, \phi_2 \in \mathcal{V} \text{ and } \forall h \in \mathcal{C}([0,T], \mathcal{H})$ we have 
    \begin{align}
        &\lim_{n \rightarrow \infty }\left[ \left( \sum_{J=1}^{N} \int_{t_2}^{t_1} \left \langle (\Lambda_{j}(v_{n}(\tau)))^{*}Q_{n}\phi_1, (\Lambda_{j}(v_{n}(\tau)))^{*}Q_{n}\phi_2 \right \rangle_{\mathbb{R}}d\tau \right) \cdot h_{v_{n|[0,t_2]}}\right]=\nonumber \\
        & \left[ \left( \sum_{J=1}^{N} \int_{t_2}^{t_1} \left \langle (\Lambda_{j}(v(\tau)))^{*}\phi_1, (\Lambda_{j}(v(\tau)))^{*}\phi_2 \right \rangle_{\mathbb{R}} d\tau \right) \cdot h_{v_{[0,t_2]}}\right]. \nonumber
    \end{align}
\end{lemma}
\begin{proof}
Let $\phi_1,\phi_2 \in \mathcal{V}$, we define
\begin{align}
    \phi_{n}= \left( \sum_{J=1}^{N} \int_{t_2}^{t_1} \left \langle (\Lambda_{j}(v_{n}(\tau,\omega)))^{*}Q_{n}\phi_1, (\Lambda_{j}(v_{n}(\tau,\omega)))^{*}Q_{n}\phi_2 \right \rangle_{\mathbb{R}} d\tau \right) \cdot h_{v_{n|[0,t_2]}} \nonumber
\end{align}

We claim that $\phi_n$ is uniformly integrable which converges $\mathbb{P}-a.s.$, to some $\phi$. To do so, it is enough to show that $r>1$ such that,
\begin{align}
    \sup_{n \geq 1} \hat{\mathbb{E}}|\phi_n|^r < \infty.  \label{claim for phi n powe r in hatE}
\end{align}
With the help of Cauchy Schwartz inequality 
\begin{align}
|(\Lambda_{j}(v_{n}(\tau,\omega))^{*}Q_{n}\phi_1|_{R} \leq |(\Lambda_{j}(v_{n}(\tau,\omega))^{*}| |Q_{n}\phi_1|_{R}  \leq  |\Lambda_{j}(v_{n}(\tau,\omega)| |\phi_1|_{\mathcal{H}}  \leq C_{12}|\phi_1|_{\mathcal{H}}. \nonumber
\end{align}
With the application of Holder's inequality
\begin{align}
    &\hat{\mathbb{E}}|\phi_n|^{r} =\hat{\mathbb{E}} \left\vert \left( \sum_{J=1}^{N} \int_{t_2}^{t_1} \left \langle (\Lambda_{j}(v_{n}(\tau)))^{*}Q_{n}\phi_1, (\Lambda_{j}(v_{n}(\tau)))^{*}Q_{n}\phi_2 \right \rangle_{\mathbb{R}} d\tau \right) \cdot h_{v_{n|[0,t_2]}}\right \vert^r\nonumber\\
    &\leq |h|^{r}_{\mathcal{L}^\infty} \hat{\mathbb{E}}\left(  \sum_{j=1}^{N} \int_{0}^{T} \left\vert (\Lambda_{j}(v_{n}(\tau)))^{*}Q_{n}\phi_1\right\vert \cdot \left\vert (\Lambda_{j}(v_{n}(\tau)))^{*}Q_{n}\phi_1\right\vert d\tau \right) \nonumber \\
    &\leq (NC^2_{12})^{r} |h|^{r}_{\mathcal{L}^\infty} |\phi_1|^r_{\mathcal{H}} |\phi_2|^r_{\mathcal{H}} T^r \nonumber
\end{align}
Hence, our claim is valid, and \eqref{claim for phi n powe r in hatE} holds consequently $\phi_n$ is uniformly integrable, which converges $\mathbb{P}-a.s.$, to some $\phi$.

Next, we will show the point-wise convergence. Let $\omega \in \hat{\Omega}$
\begin{align}
    &\lim_{n \rightarrow \infty }   \int_{t_2}^{t_1}  \sum_{J=1}^{N} \left \langle (\Lambda_{j}(v_{n}(\tau)))^{*}Q_{n}\phi_1, (\Lambda_{j}(v_{n}(\tau)))^{*}Q_{n}\phi_2 \right \rangle_{\mathbb{R}}d\tau \nonumber \\ &=
     \int_{t_2}^{t_1} \sum_{J=1}^{N} \left \langle (\Lambda_{j}(v(\tau)))^{*}\phi_1, (\Lambda_{j}(v(\tau)))^{*}\phi_2 \right \rangle_{\mathbb{R}} d\tau \label{pintwise conv of Quad }
\end{align}
fix $\omega \in \hat{\Omega}$
\begin{enumerate}
    \item $v_{n}(\cdot,\omega) \rightarrow v(\cdot,\omega)$
    \item and the $v_{n}(\cdot,\omega)_{n\geq 1}$ is uniformly bounded in $\mathcal{L}^2([0,T],\mathcal{V})$.
\end{enumerate}
To prove \eqref{pintwise conv of Quad }, it is enough to show that 
\begin{align}
    (\Lambda_{j}(v_{n}(\tau,\omega)))^{*}Q_{n}\phi_1 \longrightarrow (\Lambda_{j}(v(\tau,\omega)))^{*}\phi_1 \nonumber 
\end{align}
in $\mathcal{L}^2([t_2,t_1],\mathbb{R})$.
With the help of Cauchy-Schwarz inequality, one can achieve

\begin{align}
& \int_{t_2}^{t_1} \left \vert (\Lambda_{j}(v_{n}(\tau,\omega)))^{*}Q_{n}\phi_1 - (\Lambda_{j}(v(\tau,\omega)))^{*}\phi_1  \right \vert^2_{\mathbb{R}} d\tau \nonumber \\
&=\int_{t_2}^{t_1} \left \vert (\Lambda_{j}(v_{n}(\tau,\omega)))^{*}Q_{n}\phi_1 - (\Lambda_{j}(v_{n}(\tau,\omega)))^{*}\phi_1+ (\Lambda_{j}(v_{n}(\tau,\omega)))^{*}\phi_1- (\Lambda_{j}(v(\tau,\omega)))^{*}\phi_1  \right \vert^2_{\mathbb{R}} d\tau \nonumber\\
&=\int_{t_2}^{t_1} \left \vert (\Lambda_{j}(v_{n}(\tau,\omega)))^{*} \left(Q_{n}\phi_1 -\phi_1 \right)+ \left((\Lambda_{j}(v_{n}(\tau,\omega)))- (\Lambda_{j}(v(\tau,\omega))) \right)^{*}\phi_1 \right \vert^2_{\mathbb{R}} d\tau \nonumber\\
&\leq\int_{t_2}^{t_1} \left[  \vert (\Lambda_{j}(v_{n}(\tau,\omega)))^{*} \left(Q_{n}\phi_1 -\phi_1 \right) \vert_{\mathbb{R}} +\vert \left((\Lambda_{j}(v_{n}(\tau,\omega)))- (\Lambda_{j}(v(\tau,\omega))) \right)^{*}\phi_1  \vert_{\mathbb{R}} \right ]^2d\tau \nonumber\\
&\leq 2 \int_{t_2}^{t_1}  \vert \Lambda_{j}(v_{n}(\tau,\omega))\vert^2 \vert Q_{n}\phi_1 -\phi_1\vert^2_{\mathcal{H}} d\tau + 2 \int_{t_2}^{t_1} \vert (\Lambda_{j}(v_{n}(\tau,\omega)))- (\Lambda_{j}(v(\tau,\omega))\vert^2 \vert \phi_1  \vert^2_{\mathcal{H}} d\tau
\end{align}
Set
\begin{align}
&P^1_n=\int_{t_2}^{t_1}  \vert \Lambda_{j}(v_{n}(\tau,\omega))\vert^2 \vert Q_{n}\phi_1 -\phi_1\vert^2_{\mathcal{H}} d\tau \nonumber \\
&P^2_n=  \int_{t_2}^{t_1} \vert (\Lambda_{j}(v_{n}(\tau,\omega)))- (\Lambda_{j}(v(\tau,\omega))\vert^2 \vert \phi_1  \vert^2_{\mathcal{H}} d\tau. \nonumber
\end{align}
Since
\begin{align}
 \lim_{n \rightarrow\omega \infty } |Q_n \phi_1 - \phi_1|_{\mathcal{H}}=0 \label{ first for Pn1}  
\end{align}
and 
\begin{align}
    |B_j(v_n)| \leq C_{12}. \label{2nd for Pn1}
\end{align}
From \eqref{ first for Pn1} and \eqref{2nd for Pn1} it follows that,
\begin{align}
    \lim_{n \rightarrow \infty} P^1_{n}=0. \nonumber
\end{align}
Next from \cite{brzezniak2020global},
\begin{align}
    |\Lambda_{j}(v_n)-\Lambda_{j}(v)|\leq ||f_j||\left(||v_n||+||v||\right)||v_n-v||.\nonumber
\end{align}
Hence, it follows that,
\begin{align}
    &\int_{t_2}^{t_1} \vert (\Lambda_{j}(v_{n}(\tau,\omega)))- (\Lambda_{j}(v(\tau,\omega))\vert^2 \vert \phi_1  \vert^2_{\mathcal{H}} d\tau \nonumber\\
    &\leq |\phi_1|_{\mathcal{H}} \int_{t_2}^{t_1}||f_i||\left(||v_n(\tau,\omega)||+||v(\tau,\omega)||\right) ||v_n(\tau,\omega)-v(\tau,\omega)|| d\tau. \nonumber 
\end{align}
Using \eqref{un in V} and $v(n)(\cdot,\omega) \rightarrow v(\cdot,\omega)$ we have
\begin{align}
    \int_{t_2}^{t_1} \vert (\Lambda_{j}(v_{n}(\tau,\omega)))- (\Lambda_{j}(v(\tau,\omega))\vert^2 \vert \phi_1  \vert^2_{\mathcal{H}} d\tau \rightarrow 0. \nonumber
\end{align}
or 
\begin{align}
   \lim_{n \rightarrow \infty } P^2_{n}=0. \nonumber
\end{align}
This finishes the proof of the lemma.
\end{proof}
By applying lemma \eqref{mns1-mns2 lemma used for passing limit}, one can easily pass the limit in \eqref{firt req for quad limit pass}.\\
By applying lemma \eqref{2nd lemma for passing limit} and lemma \eqref{3rd for passing limit in quad variations}, one can pass limit to \eqref{2nd req for quad limit pass}. \\After passing limits, we conclude that  
\begin{align}
    \hat{\mathbb{E}}\left[ \left \langle \hat{M}(t_1)-\hat{M}(t_2) , \phi_1 \right \rangle_{\mathcal{H}}h(v_{|[0,t_2]}) \right]=0. \nonumber
\end{align}
And
\begin{align}
        &\hat{\mathbb{E}} \left[ \left( \left \langle \hat{M}(t_1), \phi_1 \right \rangle_{\mathcal{H}}\left \langle \hat{M}(t_1), \phi_2 \right \rangle_{\mathcal{H}}- \left \langle \hat{M}(t_2), \phi_1 \right \rangle_{\mathcal{H}}\left \langle \hat{M}(t_2), \phi_2 \right \rangle_{\mathcal{H}}\right)h(v_{|[0,t_2]})\right] \nonumber \\ &-\hat{\mathbb{E}}\left[\left( \sum_{j=1}^{m} \int_{t_2}^{t_1} \left \langle \Lambda_{j}(v(\tau)) \phi_1, \Lambda_{j}(v_{n}(\tau)) \phi_2 \right \rangle_{\mathcal{H}} d\tau \right) h(v_{|[0,t_2]}) \right]=0. \nonumber
\end{align}
As a result, we have the following immediate corollary.
\begin{corollary}
    For all $t$ in $ [0,T]$,
    \begin{align} 
        \left \langle\left \langle \hat{M} \right \rangle \right \rangle_{t}=\int_{0}^{t} \sum_{j=1}^{N} |\Lambda_{j}(v(t'))|_{\mathcal{H}}^2dt'. \nonumber
    \end{align}
\end{corollary}

\section{Existence of Martingale Solution}
Finally, we present the proof of our targeted result stated in  theorem \eqref{Martingale theorem},\\
\begin{theorem} \label{Martingale theorem}
    There exists Martingales solution to problem \eqref{main prob}.
\end{theorem}
\begin{proof}

We will follow the same procedure as Da Prato and Zabczyk in \cite{da2014stochastic} and Gaurav Dhariwal in \cite{dhariwal2017study}.
From lemma \eqref{continuity of M hat} and corollary \eqref{corollary used in proving main theorem}, we conclude that we have $\hat{M}(t)$, $t$ in $[0, T]$ which is $\mathcal{H}$-valued continuous square integrable martingale w.r.t filtration $(\mathcal{F}_t)$. Further, it has quadratic variation as given below
\begin{align*}
        \left \langle\left \langle M \right \rangle \right \rangle_{t}=\int_{0}^{t} \sum_{j=1}^{N} |\Lambda_{j}(v_{n}(t'))|_{\mathcal{H}}^2dt'.   
\end{align*}
Therefore, applying Martingale representation theorem, see \eqref{Martingale Repr THeorem}, we can find 
\begin{enumerate}
    \item A stochastic basis $(\hat{\hat{\Omega}},\hat{\hat{\mathcal{F}}},\hat{\hat{\mathcal{F}}}_t,\hat{\hat{\mathbb{P}}})$
    \item A $\mathbb{R}^N-$ valued, $\hat{\hat{\mathcal{F}}}$-Wiener process $\hat{\hat{W}}(t)$.
    \item A progressively measurable process $\hat{\hat{u}}$ such that for all $t$ in $[0,T]$ and $\omega$ in $ \mathcal{V}$ that satisfies,
    \begin{align*}
        &\left \langle \hat{\hat{u}}(t),\omega \right \rangle-\left \langle \hat{\hat{u}}(0),\omega \right \rangle \nonumber \\ &=\int_{0}^{t} \left \langle \Delta \hat{\hat{u}}(t'),\omega \right \rangle dt'+\int_{0}^{t}\left \langle F(\hat{\hat{u}}(t'),\omega \right \rangle dt'+\frac{1}{2} \sum_{j=1}^{N} \int_{0}^{t} \left \langle k_{j}(\hat{\hat{u}}(t')), w \right \rangle dt'+ \sum_{j=1}^{N} \int_{0}^{t} \left \langle \Lambda_{j}(\hat{\hat{u}}(t')), w \right \rangle d\hat{\hat{W}}(t').
    \end{align*}
\end{enumerate}
Thus, we are in the assumption of the definition of the existence of martingale solution to \eqref{main prob}, and we have martingale solution to \eqref{main prob}.
\end{proof}



\end{document}